\documentclass[english,reqno,11pt]{amsart}
\usepackage{amsmath, amsthm, amsfonts}
\usepackage{hyperref}
\hypersetup{
	colorlinks=true,
		citecolor=blue!60!black,
		linkcolor=red!60!black,
		urlcolor=green!40!black,
		filecolor=yellow!50!black,
	breaklinks=true,
	pdfpagemode=UseNone,
	bookmarksopen=false,
}
\usepackage{amsmath,amsthm,amssymb}
\usepackage{amscd,indentfirst,epsfig}
\usepackage{latexsym}
\usepackage{times}
\usepackage{enumerate}
\usepackage{mathrsfs}
\usepackage{stmaryrd}
\usepackage{amsopn}
\usepackage{amsmath}
\usepackage{amssymb,dsfont,mathtools}
\usepackage{amsfonts,bm}
\usepackage{amsbsy,amsmath}
\usepackage{amscd}
\usepackage{xcolor}
\usepackage{mathtools}
\usepackage{cancel}
\usepackage{hyperref} 
\usepackage{mathrsfs}
\usepackage{amsmath,amsthm,amssymb}
\usepackage{latexsym}
\usepackage{latexsym}
\usepackage{bm}
\usepackage{enumerate}
\usepackage{mathrsfs}
\usepackage{stmaryrd}
\usepackage{amsopn}
\usepackage{amsmath}
\usepackage{amssymb}
\usepackage{amsfonts}
\usepackage{amsbsy}
\usepackage{amscd,indentfirst}
\usepackage{hyperref}
\usepackage{amsfonts,amsmath,latexsym,amssymb,verbatim,amsbsy}
\usepackage{amsthm}
\usepackage{colordvi}
\usepackage{dsfont}
\usepackage{hyperref}
\hypersetup{
	colorlinks=true,
		citecolor=blue!60!black,
		linkcolor=red!60!black,
		urlcolor=green!40!black,
		filecolor=yellow!50!black,
	breaklinks=true,
	pdfpagemode=UseNone,
	bookmarksopen=false,
}
\usepackage{amsmath,amsthm,amssymb}
\usepackage{amscd,indentfirst,epsfig}
\usepackage{latexsym}
\usepackage{times}
\usepackage{enumerate}
\usepackage{mathrsfs}
\usepackage{stmaryrd}
\usepackage{amsopn}
\usepackage{amsmath}
\usepackage{amssymb,dsfont,mathtools}
\usepackage{amsfonts,bm}
\usepackage{amsbsy,amsmath}
\usepackage{amscd}
\usepackage{xcolor}

\usepackage[mono=false]{libertine}
\usepackage[T1]{fontenc}
\usepackage{amsthm}
\usepackage{cancel}
\usepackage[cal=euler, scr=boondoxo]{}
\usepackage{microtype}

\usepackage{numprint}

\makeatletter
\def \leq {\leqslant}
\def \le {\leq}
\def \geq {\geqslant}

\def \ge {\geq}

\def\R{\mathbb R}

\def\N{\mathbb N}
\def\C{\mathscr{C}}
\def\D{\mathcal D}

\def \M {\mathcal{M}}

\def\g{\gamma}
\def \ds {\displaystyle}

\def \d {\mathrm{d}}

\def \Q {\mathcal{Q}}

 \def \w {\bm{w}}

\numberwithin{equation}{section}

\newcommand{\emptylabel}[1]{}

\newcommand{\abs}[1]{\left\vert#1\right\vert}

\def\:{\colon}

\def\C{\mathbb{C}}
\def\N{\mathbb{N}}
\def\P{\mathbb{P}}
\def\R{\mathbb{R}}

\def\D{\mathcal{D}}

\def\d{\,\mathrm{d}}
\def\dx{\d x}

\def\dy{\d y}

\def\M{\mathcal M} 
\def\P{\mathcal P}
\def\p{\partial}

\def\ir{\int_{\R}}

\numberwithin{equation}{section}

\usepackage{natbib}
\newtheorem{theo}{Theorem}[section]
\newtheorem{cor}[theo]{Corollary}

\newtheorem{lem}[theo]{Lemma}
\newtheorem{prp}[theo]{Proposition}

\newtheorem{rem}[theo]{Remark}
\newtheorem{defi}[theo]{Definition}
\newcommand{\vertiii}[1]{{\left\vert\kern-0.25ex\left\vert\kern-0.25ex\left\vert #1  
    \right\vert\kern-0.25ex\right\vert\kern-0.25ex\right\vert}}                      
\newcommand{\verti}[1]{{\left\vert\kern-0.25ex\left\vert\kern-0.25ex\left\vert #1    
    \right\vert\kern-0.25ex\right\vert\kern-0.25ex\right\vert}}						 

\title[Relaxation of the 1D dissipative Boltzmann]{Relaxation in Sobolev spaces and $L^1$ spectral gap of the 1D dissipative Boltzmann equation with Maxwell interactions}

\def\theauthor{R. Alonso, V. Bagland, J. A. Ca\~{n}izo, B. Lods, S. Throm}

\hypersetup{pdfauthor={\theauthor}}

\vspace{1cm}

\author{R. Alonso}
\address{Division of Arts \& Sciences, Texas A\&M University at Qatar, Education City, Doha, Qatar.}
\email{ricardo.alonso@qatar.tamu.edu}

\author{V. Bagland}
\address{Universit\'{e} Clermont Auvergne, LMBP, UMR 6620 - CNRS,  Campus des C\'ezeaux, 3, place Vasarely, TSA 60026, CS 60026, F-63178 Aubi\`ere Cedex, France.}
\email{Veronique.Bagland@uca.fr}

\author{J. A. Ca\~{n}izo}
\address{Departamento de Matem\'{a}tica Aplicada \& IMAG, Universidad de Granada, Avenida de Fuentenueva S/N, 18071 Granada, Spain.}
\email{canizo@ugr.es}

\author{B. Lods}
\address{Universit\`{a} degli Studi di Torino \& Collegio Carlo Alberto, Department of Economics, Social Sciences, Applied Mathematics and Statistics ``ESOMAS'', Corso Unione Sovietica, 218/bis, 10134 Torino, Italy.}
\email{bertrand.lods@unito.it}

\author{S. Throm}
\address{{Ume\aa} University,  Department of Mathematics and Mathematical Statistics,  901 87 Ume\aa, Sweden }
\email{sebastian.throm@umu.se}

\date{}

\begin{document}

\begin{abstract}
We study the dynamic relaxation to equilibrium of the 1D dissipative Boltzmann equation with Maxwell interactions in classical $H^s$ Sobolev spaces.  In addition, we present a spectral shrinkage analysis and spectral gap estimates for the linearised 1D dissipative Boltzmann operator with such interactions.  Based on this study, we explore the convergence in $H^s$ and $L^{1}$ spaces for the linear and nonlinear models.  This study extends classical results found in the literature given for spaces with weak topologies.
\end{abstract}

\maketitle


\section{Introduction}
\label{sec:intro}

In this work, we revisit the exponential convergence to equilibrium for the one-dimensional inelastic Boltzmann model in self-similar variables with Maxwell interactions studied in \cite{BK,bobcerc1,MR2355628}.  We provide a new detailed analysis of the nonlinear and linear problems, including spectral gap estimates, in classical $H^s$ Sobolev and weighted $L^{1}$ spaces. The results obtained in the present contribution are tailored to be used in a companion paper \cite{unique-short} in which the uniqueness of self-similar profile for the $1D$ inelastic Boltzmann equation is proved for moderately hard potentials (see particularly Section~\ref{Sec:additional}). We refer to \cite{unique-short} for more details about the one-dimensional inelastic Boltzmann model and more generally the physical relevance of inelastic kinetic equations.

\subsection{One-dimensional Boltzmann for Maxwell molecules}
We consider the following Boltzmann equation on the real line
\begin{equation}\label{eq:IB}
 \p_t f (t,x)=  \int_\R f\left(x + \frac{y}{2}\right) f\left(x - \frac{y}{2}\right) \dy
  - f(x) \int_\R f(y) \dy=\Q_{0}(f,f)
\end{equation} 
which models particles performing inelastic collisions in one dimension. We refer the reader to  \cite{BK,bobcerc1,MR2355628} for a thorough description of this model and point out that we restrict, for simplicity, to the case of \emph{sticky particles}. More precisely, this means that two particles interacting with pre-collisional velocities $x,y$ would end up with post-collisional velocities $x',y'$ given by 
\begin{equation*}
x'=y'=\frac{x+y}{2}\,.
\end{equation*}
In Eq. \eqref{eq:IB}, the (symmetrised) collision operator $\Q_{0}$ is given by
\begin{equation*}\begin{split}
  \Q_{0}(f,g)(x)
  &= \int_\R f\left(x +\frac{y}{2}\right) g\left(x - \frac{y}{2}\right) \dy
  - \frac12 f(x) \int_\R g(y) \dy
  - \frac12 g(x) \int_\R f(y) \dy
  \\
  &=: \Q_{0}^+(f,g) - \Q_{0}^-(f,g).
\end{split}\end{equation*}
In weak form the collision operator reads
\begin{equation}
  \label{eq:weak}
  \int_\R \Q_{0}(f,g)(x) \phi(x) \dx
  =
  \int_\R \int_\R f(x) g(y) \left(
    \phi\Big(\frac{x+y}{2}\Big) - \frac12\phi(x) - \frac12\phi(y)
  \right) \dx \dy
\end{equation}
for any smooth enough test function $\phi=\phi(x)$.
From \eqref{eq:weak} it follows that \eqref{eq:IB} at least formally conserves the \emph{mass} $\ds\int_\R f(t,x)\dx$ and \emph{momentum} $\ds\int_{\R}xf(t,x)\dx$ while the kinetic \emph{energy} $E(t)=\ds\int_{\R}x^2f(t,x)\dx$ is decreasing:
\begin{equation}\label{eq:Ener}
\dfrac{\d}{\d t}E(t)=-\frac{1}{4}\int_{\R^{2}}f(t,x)f(t,y)|x-y|^{2}\d x\d y\,.\end{equation} This suggests to consider the self-similar change of variables 
\begin{equation*}
 g(t,x)=\sqrt{E(t)}f(t,\sqrt{E(t)}x)
\end{equation*}
which fixes the energy of $g$ to one.  {Equation \eqref{eq:IB} is complemented by an initial condition $f(0,x)=f_{0}(x)$ for which we can assume, exploiting scale invariances, without loss of generality that
\begin{equation}
  \label{eq:normalisation-f0}
  \int_{\R} f_0(x) \dx = 1,
  \qquad
  \int_{\R} xf_0(x) \dx = 0.
\end{equation}
This yields conservation of mass and  momentum again at least formally, i.e.\@
\begin{equation*}
  \int_{\R} f(t,x) \dx = 1,
  \qquad
  \int_{\R} xf(t,x) \dx = 0, \qquad \text{for all } t\geq 0.
  \end{equation*}
It then follows that \eqref{eq:Ener} reads $\frac{\d}{\d t}E(t)=-\frac12 E(t)$ and $g$ satisfies
\begin{equation}
  \label{eq:IB-selfsim}
  \p_t g = -\frac14 \p_x (xg) + \Q_{0}(g,g)
\end{equation}
to which we refer as the \emph{self-similar equation} for Maxwell molecules. We complement \eqref{eq:IB-selfsim} with the initial condition $g(0,x)=g_{0}(x):=\sqrt{E_0} f_{0}(\sqrt{E_0}x)$ with $E_0=\int_{\R} x^2 f_{0}(x) \d x$. We deduce from \eqref{eq:normalisation-f0} and the definition of $g_{0}$  that
\begin{equation}
 \label{eq:normalisation-g0}
  \int_{\R} g_0(x) \dx = 1,
  \qquad
  \int_{\R} xg_0(x) \dx = 0,
  \qquad
  \int_{\R} x^2 g_0(x) \dx = 1.
\end{equation}
Now,  mass, momentum and energy are conserved at least formally by equation \eqref{eq:IB-selfsim}, i.e.\@
\begin{equation}
  \label{eq:normalisation-g}
  \int_{\R} g(t,x) \dx = 1,
  \qquad
  \int_{\R} xg(t,x) \dx = 0,
  \qquad
  \int_{\R} x^2g(t,x) \dx = 1 \qquad \text{for all } t\geq 0
\end{equation}}
We will use the following concept of weak (measure) solutions for \eqref{eq:IB-selfsim}:
\begin{defi}\label{Def:weak:sol}
 Let $\M_{2}(\R)$ denote the set of real Borel measures on $\R$ with finite moments up to order $2$ (see \eqref{eq:Mk}). A family of non-negative measures $\mu \colon [0,\infty) \to \M_{2}(\R)$ is denoted a weak solution to \eqref{eq:IB-selfsim} if 
\begin{multline*}\label{eq:weak:solution}
\frac{\d}{\d t} \int_{\R} \phi(x)\mu(t,\!\d x)=\frac{1}{4}\int_{\R}x\phi'(x)\mu(t,\!\d x)\\
+\int_{\R\times\R}\left(\phi\left(\frac{x+y}{2}\right)-\frac{1}{2}\phi(x)-\frac{1}{2}\phi(y)\right)\mu(t,\!\dx)\mu(t,\!\dy) \qquad \forall \phi \in \mathcal{C}_{b}^{1}(\R).
\end{multline*}
If the left-hand side is zero, i.e. if $\mu$ does not depend on $t$, $\mu$ is denoted a stationary or steady solution.
\end{defi}

\medskip

Our main goal of this work consists in studying the approach to a self-similar profile for \eqref{eq:IB} which equivalently corresponds to convergence to stationary states of \eqref{eq:IB-selfsim}. An important property of the Maxwell molecules case is that, due to explicit computations in Fourier variables, these steady solutions can actually be given explicitly. In fact one has the following statement.

\begin{theo}[\cite{bobcerc1}]\label{theo:bob} Any stationary weak solution $\mu \in \M_{2}(\R)$ of \eqref{eq:IB-selfsim}, with $\M_{2}(\R)$ given in \eqref{eq:Mk}, such that 
$$\int_{\R}\mu(\!\dx)=1,\qquad \int_{\R}x\mu(\!\dx)=0, \qquad \int_{\R}x^{2}\mu(\!\d x)=\frac{1}{\lambda^{2}} >0$$
is of the form
$$\mu(\!\dx)=H_{\lambda}(x)\d x=\lambda\bm{H}(\lambda x)\d x$$
with
\begin{equation*}
\bm{H}(x) = \frac{2}{\pi (1+x^2)^2}\,\qquad x \in \R.
\end{equation*}
\end{theo}
In particular, $\bm{H}$ is the unique steady solution to \eqref{eq:IB-selfsim} with unit mass and energy and zero momentum. The existence and uniqueness has been obtained in \cite{bobcerc1} relying on Fourier methods which has been extended to measure solutions in \cite{MR2355628}. 

\subsection{Notations}
\label{sec:notation}
Before stating our main results, let us collect some notation used throughout this work. For the weight function
\begin{equation*}
\w_{a}(x)=\left(1+|x|\right)^{a}, \qquad a\in \R, \qquad x \in \R
\end{equation*}
we denote the corresponding weighted Lebesgue space $L^{1}(\w_{a})$ by
$$L^{1}(\w_{a}) :=\Big\{f : \R \to \R\;;\,\|f\|_{L^{1}(\w_{a})}:=\int_{\R}\big|f(x)\big|^{}\,\w_{a}(x)\,\d x  < \infty\Big\}\,.$$
For $a=0$, we simply denote $L^{1}(\w_{0})=L^{1}(\R)$ and $\|\cdot\|_{L^{1}}=\|\cdot\|_{L^{1}(\w_{0})}.$ More generally, for any $1 \leq p \leq \infty$, $\|\cdot\|_{L^{p}}$ will denote the standard norm on the Lebesgue space $L^{p}(\R)$.
For $f \in L^{1}(\w_{a}))$ we also denote by 
$$M_{a}(f) :=\int_{\R}f(x)\,  |x|^a \d x $$ the corresponding moment of order $a \in \R$. For general $s\ge 0$, we define the fractional \emph{homogeneous} Sobolev space $\dot{H}^s(\R)$ as the space of tempered distribution $f:\R \to \R$ with Fourier transform $\widehat{f} \in L^{1}_{\mathrm{loc}}(\R)$ and such that
$$\|f\|_{\dot{H}^{s}}=\left(\int_{\R} |\xi|^{2s}|\widehat{f}(\xi)|^{2}\d \xi\right)^{\frac{1}{2}} <\infty.$$
Here the Fourier transform of $f$ is defined as
$$ \widehat{f}(\xi) =\int_{\R} f(x) e^{-ix\xi}\d x, \qquad \xi \in \R. $$
In the same way, we define the Sobolev space $H^{s}(\R)$ as
$$H^s(\R):=\Big\{f \in L^{2}(\R)\;;\; \|f\|_{H^{s}}:=\left(\int_{\R} (1+|\xi|^2)^s |\widehat{f}(\xi)|^2\d \xi\right)^{\frac{1}{2}}<\infty\Big\}\,.$$
Due to the conservation of mass and momentum by \eqref{eq:IB} and additionally energy by \eqref{eq:IB-selfsim}, it is natural to introduce the following subspaces of $L^{1}(\w_{a}):$
\begin{multline}\label{eq:spaces}
{\mathbb{Y}}_{a}=\left\{f \in L^{1}(\w_{a})\;;\;\int_{\R}f(x)\d x=\int_{\R}f(x) x\d x=0\right\}\\
\text{ and }\qquad {\mathbb{Y}}_{a}^{0}=\left\{f \in {\mathbb{Y}_{a}}\;;\;\int_{\R}f(x) x^{2}\d x=0\right\}
\end{multline}
for $2 < a <3$ equipped with the norm $\|\cdot\|_{L^{1}(\w_{a})}$. Similarly, we define a suitable space of measures reflecting the conserved quantities. For $k\geq 2$, let ${\mathcal M}_{k}(\R)$ be the set of real Borel measures on $\R$ with moments up to order $k$, i.e.
\begin{equation}\label{eq:Mk}
 {\mathcal M}_{k}(\R)=\Bigl\{\text{real Borel measures $\mu$ on } \R \,;\, \int_{\R}  \w_{k}(x) \,|\mu|(\!\dx)<\infty\Bigr\}.
\end{equation}
We then denote
{\begin{equation}\label{eq:X0:measure}
X_{k}:=  \left\{ \mu\in {\mathcal M}_{k}(\R) \;\Bigg| \Bigg. \begin{array}{l}
  \displaystyle \int_{\R} \mu(\!\d x) = \int_{\R} x\, \mu(\!\d x) = \int_{\R} x^2 \,\mu(\!\d x) = 0
\end{array}\right\}. 
\end{equation}
We can equip $X_k$ with various norms based on Fourier variables. More precisely, for $k\geq 0$, the space of continuous functions $\psi \: \R \to \C$ such that $\xi \mapsto \psi(\xi) \,|\xi|^{-k}$ converges to a limit as $\xi \to 0$ equipped with the norm
\begin{equation}\label{eq:fourier:norm:1}
 \vertiii{\psi}_{k} := \sup_{\xi \in \R \setminus \{0\}} \frac{|\psi(\xi)|}{|\xi|^k}
\end{equation}
is a Banach space. The same construction can be extended to $L^p$ norms in Fourier variables
\begin{equation}\label{def:normkp}
 \vertiii{\psi}_{k,p}^p := \ir \frac{|\psi(\xi)|^p}{|\xi|^{kp}} \d \xi,
\end{equation}
for $k\geq 0$ and $p\in(1,\infty)$. The integral is finite provided $|\psi(\xi)| \leq \min\{1 ,C|\xi|^3\}$ for some $C>0$  and
$\frac{1}{p} < k < 3 + \frac{1}{p}$. For $\mu\in X_k$ it follows from \cite[Proposition 2.6]{MR2355628} that $\vertiii{\widehat{\mu}}_{k}<\infty$ for any $0< k< 3$. In particular, $X_{k}$ is a Banach space when equipped with the norm $\vertiii{\cdot}_{k}$ for $2< k<3,$ (see Proposition 2.7 in \cite{MR2355628}). We refer to Lemma \ref{lem:mukvertk} for more properties of the norms $\vertiii{\cdot}_{k},\vertiii{\cdot}_{k,p}$.

\medskip

\subsection{Main results}

In following we describe the main results we obtain in this work which can be divided in two categories: first, we consider the non-linear problem \eqref{eq:IB-selfsim} and derive explicit convergence rates in the Sobolev spaces $H^s(\R)$. Furthermore, we study the corresponding linearised equation and obtain suitable estimates on the spectral gap for the respective linear collision operator in self-similar variables.

\subsubsection{Convergence to equilibrium for the non-linear problem}

Our first main result concerns the convergence to equilibrium for \eqref{eq:IB-selfsim} with respect to Fourier based metrics. In particular this extends previous results of \cite{MR2355628} relying on the fact that \eqref{eq:IB-selfsim} in Fourier variables simplifies to a (non-local) ODE (see \eqref{eq:selfsim-fourier}).

\medskip

 \begin{theo}\phantomsection\label{k-norm-cvgce}
  Assume that $g = g(t,x)$ is a non-negative solution to
  \eqref{eq:IB-selfsim} with the normalisation
  \eqref{eq:normalisation-g}. Then, for $0 \leq k < 3$, and
  for all $t \geq 0$,
  \begin{equation*}
   \vertiii{\widehat{g}(t,\cdot)  - \widehat{\bm{H}}}_{k} \leq \exp\left(-\sigma_{k} t\right) \vertiii{\widehat{g}_0 - \widehat{\bm{H}}}_{k}\,, \qquad  \qquad \sigma_{k} := 1 - \frac14 k - 2^{1-k}\,.
  \end{equation*}
In particular, $g(t)$ converges exponentially to $\bm{H}$ in the $k$-Fourier norm for any $2 < k < 3$. More generally, for  $p\ge1$,  $\frac{1}{p} < k < 3+\frac1p$, and
  for all $t \geq 0$,
  \begin{equation*}
    \vertiii{\widehat{g}(t,\cdot)- \widehat{\bm{H}}}_{k,p} \leq \exp\left(-\sigma_{k}(p) t\right) \vertiii{\widehat{g}_0 - \widehat{\bm{H}}}_{k,p}\,,
    \qquad \qquad \sigma_{k}(p):= 1 - \frac14 k + \frac{1}{4p} - 2^{1 + \frac{1}{p} - k}\,.
  \end{equation*}
  In particular, $g(t)$ converges exponentially to $\bm{H}$ in the $k$-Fourier
  norm for any $(k,p)$ such that $\sigma_{k}(p)>0$.
\end{theo}
The first part of the result (convergence in the norm $\vertiii{\cdot}_{k}$) is essentially contained in \cite{MR2355628} but we adopt here a simplified approach and extend the result to the new class of Fourier metrics $\vertiii{\cdot}_{k,p}$, $1 \leq p< \infty.$

Theorem~\ref{k-norm-cvgce} can be improved to obtain convergence even with respect to Sobolev norms by relying on a detailed analysis of the propagation of regularity for \eqref{eq:IB-selfsim} in Fourier variables (see Theorem~\ref{theo-baseline}). In fact, we have the following statement.

\begin{theo}[\textit{\textbf{Sobolev norm propagation and relaxation}}]\phantomsection\label{theo:Sobolev}
Let $g(t)=g(t,x)$ be a solution to the Boltzmann problem \eqref{eq:IB-selfsim}-\eqref{eq:normalisation-g0} with initial condition $g_{0}(x)=g(0,x)$ satisfying 
$$|\widehat{g_0}(\xi) | \leq \left(1+c^{2}\,|\xi|^{2}\right)^{-\frac{\beta}{2}}, \qquad \xi \in \R$$ for some $\beta,\,c>0$ and $g_0 \in H^{\ell}(\mathbb{R})$ for $\ell\geq0$.  Then, for $\frac{5}{2} < k < {3}$ and  $0<\sigma < \frac98 - \frac k4  - 2^{\frac{3}{2} - k}$ one has
\begin{equation}\label{H-relax}
  \| g(t) - \bm{H}\|_{H^{\ell}} \leq \exp\left(-\sigma t\right)\,\Big( \| g_0 - \bm{H}\|_{H^{\ell}} +  C(\sigma,\ell,k)\,\| g_0 - \bm{H}\|_{L^{1}(\w_{k})}  \Big)\,,
\end{equation}
{ for some positive constant $C(\sigma,\ell,k) >0$ depending only on $\ell,k$ and $\sigma.$}
\end{theo}
The propagation of  regularity, uniformly in time, for the rescaled equation \eqref{eq:IB-selfsim} has been investigated already in \cite{FPTT} and our result extends in particular \cite[Theorem 5]{FPTT} which was obtained  via a semi-implicit discretization of Eq. \eqref{eq:IB-selfsim}. We propose here a new approach which is more direct (no iteration/approximation step) and based purely on comparison arguments.  Our Theorem \ref{theo:Sobolev} is new and proves at the same time the propagation of Sobolev estimates and the convergence in Sobolev norms with explicit rate of convergence. By simple interpolation, the result also provides the rate of convergence in $L^{1}$ (see Corollary \ref{cor:L1conv}).

\subsubsection{Spectral gap for the linearised problem} The exponential convergence towards equilibrium in the various norms provided in Theorems \ref{k-norm-cvgce} or \ref{theo:Sobolev} strongly suggests the existence of a spectral gap for the associated linearized operator in the spaces considered in such results.  For that purpose, we introduce 

 \begin{defi}\label{def:L}
   We define the linearised operator 
 $\mathscr{L}:\mathscr{D}(\mathscr{L}) \subset X_k\to X_k$
 on the Banach space $X_k$ given in \eqref{eq:X0:measure} by
 \begin{equation*}
 \mathscr{L}h (x) := -\frac14 \p_x( x h) + 2\Q_{0}(h,\bm{H}), \quad h \in \mathscr{D}(\mathscr{L})
\end{equation*}
and $\mathscr{D}(\mathscr{L})=\{h \in X_{k}\;;\;\p_{x}(x h) \in X_{k}\}.$
\end{defi}

In this work, we prove new spectral gap estimates for $\mathscr{L}$ which correspond to convergence rates for the linearised equation
\begin{equation}
  \label{eq:linearised}
  \p_t h = \mathscr{L}h = -\frac14 \p_x( x h) + 2 \Q_{0}(h, \bm{H}),
\end{equation}
with initial condition $h(0,x)=h_{0}(x)$ satisfying
\begin{equation}
   \label{eq:normalisation-lin}
   \int_\R h_{0}(x) \dx =
   \int_\R x h_{0}(x) \dx  = \int_\R x^2 h_{0}(x) \dx = 0,
 \end{equation}
In fact, in analogy to Theorem~\ref{k-norm-cvgce} our first result provides a spectral gap estimate for $\mathscr{L}$ with respect to the Fourier norms $\vertiii{\cdot}_{k}$ and $\vertiii{\cdot}_{k,p}$:
\begin{theo}\label{specgap}
  Assume that $h = h(t,x)$ is a solution to
  \eqref{eq:linearised} with the normalisation
  \eqref{eq:normalisation-lin}. Then, for $0 \leq k < 3$
  \begin{equation*}
    \vertiii{\widehat{h}(t,\cdot)}_k \leq \exp\left(-\sigma_{k} t\right) \vertiii{\widehat{h}_0}_k 
    \qquad \forall t \geq0\,,\end{equation*}
    where $h_{0}=h(0,\cdot)$ and
$\sigma_{k}:= 1 - \frac14 k - 2^{1-k}.$  
 In particular, $h(t,\cdot)$ converges exponentially to $0$ in the $k$-Fourier norm for any $2 < k < 3$. Moreover, for any $1 \leq p < \infty$, 
\begin{equation*}
\vertiii{\widehat{h}(t,\cdot)}_{k,p} \leq \exp\left(-\sigma_{k}(p) t\right)\vertiii{\widehat{h}_0}_{k,p} \qquad \forall t \geq0,\end{equation*}
where $\sigma_{k}(p)=1 - \frac14 k + \frac{1}{4p} + 2^{1 + \frac{1}{p} - k}.$ 
\end{theo}
The existence of a spectral gap for $\mathscr{L}$ in the space $X_{k}$ endowed with the Fourier norm $\vertiii{\cdot}_{k}$ is essentially well-known (see e.g. \cite{MR2355628}) but we revisit the arguments in Section \ref{sec:sg-small}.   For practical purpose, and in particular for the stability of the spectral properties of the linearized operator associated to moderate hard potentials as considered in the companion paper \cite{unique-short}, it is important to extend the existence of an explicit spectral gap in spaces of the type $L^{1}(\w_{a})$. As a main contribution, similarly to Theorem~\ref{theo:Sobolev}, the result from Theorem \ref{specgap} can be transferred to the more tractable class of spaces ${\mathbb{Y}_{a}^{0}}$ (see \eqref{eq:spaces}) with $2<a<3$. More precisely, for $\mu\in X_k$ \cite[Lemma 2.5 and Proposition 2.6]{MR2355628} imply
$${\vertiii{\widehat{\mu}}_{k}}\leq C \int_{\R} \w_{k}(x) \,|\mu|(dx) \qquad \mbox{ for any  }\; 2< k<3,$$ 
with $\w_{k}(x)=(1+|x|)^{k}$ and thus $\mathbb{Y}_{a}^{0}\subset X_k$ for $a\geq k$. To restrict the spectral gap from the larger space $X_k$ to $\mathbb{Y}_a^{0}$ we can rely on techniques developed in \cite{GMM} for the opposite procedure, i.e. extension of a spectral gap to a larger space (see also \cite{MM} for pioneering work on the shrinkage as well as \cite{CanizoThrom}). More precisely, one exploits a suitable splitting of the operator
$$\mathscr{L}_{0}=A+B,$$ with $A: X_k \to \mathbb{Y}_{a}^{0}$ bounded
and $B$ enjoying some dissipative properties (we refer to Section \ref{sec:sg-small} for more details). This leads to the following statement.

\begin{theo}\label{restrict}
  Let $2<a<3$. The operator $(\mathscr{L},\mathscr{D}(\mathscr{L}))$ generates a strongly continuous semigroup $\left(S_{0}(t)\right)_{t\geq0}$ on $\mathbb{Y}_{a}^{0}$ and for any $ {\nu \in(0,1-\frac{a}{4} -2^{1-a})}$, there exists $C(\nu)>0$ such that
  $$ \|S_{0}(t)h\|_{L^1(\w_{a})}\le C(\nu) \exp\left(-\nu t\right) \|h\|_{L^1(\w_{a})}$$
  for any $h\in \mathbb{Y}_{a}^{0}$ and $t\ge 0$. Moreover, one has that
  $$\|\mathscr{L}h\|_{L^1(\w_a)}\ge \frac{\nu}{C(\nu)} \|h\|_{L^1(\w_a)}, \qquad \text{for any } h \in \mathscr{D}(\mathscr{L}) {\cap\mathbb{Y}_{a}^{0}}.$$
\end{theo}

 \subsection{Organisation of the paper}
 
The remainder of this work is structured as follows. In Section~\ref{Sec:conv:eq:fourier} we prove the convergence to equilibrium for the non-linear problem \eqref{eq:IB-selfsim} in Fourier norms as stated in Theorem~\ref{k-norm-cvgce}. Section~\ref{sec:regularity} is devoted to the proof of Theorem~\ref{theo:Sobolev} by providing detailed estimates on the propagation of regularity in Fourier variables. In Section~\ref{Sec:linear} we will prove Theorems \ref{specgap} and \ref{restrict} by extensively studying the linear problem and the corresponding linearised operator. Moreover, we discuss various additional consequences  (see Section~\ref{Sec:additional}) which are particularly relevant in the companion article \cite{unique-short}. Finally, an auxiliary result related to the Fourier norms is provided in the Appendix~\ref{appendix}.

 \subsection*{Acknowledgments}

R.~Alonso gratefully acknowledges the support from O Conselho Nacional de Desenvolvimento Cient\'ifico e Tecnol\'ogico, Bolsa de Produtividade em Pesquisa - CNPq (303325/2019-4). J.~Ca\~{n}izo acknowledges support from grant
PID2020-117846GB-I00, the research network RED2018-102650-T, and the
Mar\'ia de Maeztu grant CEX2020-001105-M from the Spanish
government. B.~Lods was partially supported by PRIN2022 (project ID: BEMMLZ) Stochastic control and games and the role of information. He also gratefully acknowledges the financial support from the Italian Ministry of Education, University and Research (MIUR), Dipartimenti di Eccellenza grant 2022-2027
as well as the support from the de Castro Statistics Initiative, Collegio Carlo Alberto (Torino).  The authors would like to acknowledge the support of the
Hausdorff Institute for Mathematics where this work started during
their stay at the 2019 Junior Trimester Program on Kinetic Theory.

\medskip
\emph{Data sharing not applicable to this article as no datasets were generated or analysed during the current study.}

\section{Exponential convergence to equilibrium in Fourier norms -- Proof of Theorem~\ref{k-norm-cvgce}}\label{Sec:conv:eq:fourier} 
In this section, we will give the proof of Theorem~\ref{k-norm-cvgce} relying on the representation of \eqref{eq:IB-selfsim} in Fourier variables. Notice that the gain term $\Q_{0}^+$ of the collision operator can be alternatively written as
\begin{equation*}\begin{split}
  \Q_{0}^+(f,g)(x) &= \int_\R f\left(x + \frac{y}{2}\right) g\left(x - \frac{y}{2}\right) \dy
  = 2 \int_\R f(x + y) g(x - y) \dy
  \\
 & = 2 \int_\R f(y) g(2x - y) \dy
  = 2 (f\ast g)(2x).
\end{split}\end{equation*}
This convolution nature of the collision operator makes the formulation of \eqref{eq:IB-selfsim} in terms of the  Fourier transform of $g(t)$ natural. Precisely, the Fourier transform $\varphi$ of $g$ given by
\begin{equation*}
  \varphi(t, \xi) := \int_{\R} g(t,x) e^{-ix \xi} \dx, \quad \xi \in \R
\end{equation*}
solves
\begin{equation}
  \label{eq:selfsim-fourier}
  \p_t \varphi(t,\xi) = \frac14 \,\xi  \,\p_\xi \varphi(t,\xi)
  + \varphi\Big(t, \frac{\xi}{2} \Big)^2 - \varphi(t,\xi), 
\end{equation}
with the initial condition $\varphi(0,\cdot)=\widehat{g}_{0}=:\varphi_0$.
Due to \eqref{eq:normalisation-g}, $\varphi$ satisfies for all
$t \geq 0$ that
\begin{equation}
  \label{eq:normalisation-varphi}
  \varphi(t,0) = 1,
  \qquad
  \p_\xi \varphi(t,0) = 0,
  \qquad
  \p_\xi^2 \varphi(t,0) = -1.
\end{equation}
Notice that the unique steady state $\bm{\Phi}$ of \eqref{eq:selfsim-fourier} satisfying \eqref{eq:normalisation-varphi} is given (see \cite{bobcerc1}) by
\begin{equation}\label{eq:Phi}
  \bm{\Phi}(\xi) = (1 + |\xi|) \exp\left(-|\xi|\right)\,, \qquad \xi \in\R.
\end{equation}
 which is exactly the Fourier transform of $\bm{H}$. With these observations, we can give the proof of Theorem \ref{k-norm-cvgce}.
 \begin{proof}[Proof of Theorem~\ref{k-norm-cvgce}] Assume that $g=g(t,x)$ is a solution to \eqref{eq:IB-selfsim} with the normalisation
\eqref{eq:normalisation-g}, and call $\varphi = \varphi(t,\xi)$ its
Fourier transform as before. Then $\varphi(t,\cdot)$ is a solution to
\eqref{eq:selfsim-fourier} with the normalisation
\eqref{eq:normalisation-varphi}, and we may take the difference with $\bm{\Phi}$ given in \eqref{eq:Phi} to estimates $\vertiii{\varphi(t)-\bm{\Phi}}_{k,p}$. We begin with the first part of the proof, corresponding to the special case $p=\infty.$\\

\noindent \textit{$\bullet$ The case $p=\infty$.}  We define $\psi(t,\xi)
:= \varphi(t,\xi) - \bm{\Phi}(\xi),$ for $t \geq 0$. Then, $\psi$ satisfies
\begin{equation}\label{nlinear-self-similar}
  \p_t \psi(t,\xi) =
  \frac14 \xi  \p_\xi \psi(t,\xi)
  + \psi\Big(t, \frac{\xi}{2} \Big)
  \left( \varphi\Big(t, \frac{\xi}{2} \Big)
    + \bm{\Phi}\Big(\frac{\xi}{2}
    \Big)  \right)
  - \psi(t,\xi).
\end{equation}
If we call $\left(T(t)\right)_{t\geq0}$ the semigroup associated to the operator $\psi \mapsto \frac14 \xi
\p_\xi \psi - \psi$ is given by
\begin{equation*}
  T(t) \phi(\xi) := e^{-t} \phi( \xi e^{\frac{1}{4}t}), \qquad t \geq0, \xi \in \R\,.
\end{equation*}
Then by Duhamel's formula, the solution $\psi(t)=\psi(t,\xi)$ can be written as 
\begin{equation}
  \label{eq:decay-Duhamel}
  \psi(t) = T(t) \psi_0 + \int_0^t T(t-s) A(s) \d s,\,
\end{equation}
where \begin{equation*}
  A(s)=A(s,\xi) := \psi\Big(s, \frac{\xi}{2} \Big)
  \left( \varphi\Big(s, \frac{\xi}{2} \Big)
    + \bm{\Phi}\Big(\frac{\xi}{2}
    \Big)  \right).
\end{equation*}
Now we notice that, for any $h$ such that $\vertiii{h}_{k}$ is finite,
\begin{equation}
  \label{eq:decay1}
  \vertiii{T(t) h}_{k}
  = e^{-t} \sup_{\xi \neq 0} \frac{|h(\xi e^{\frac14 t})|}{|\xi|^k}
  = e^{-(1 - \frac14 k)t}
  \sup_{\xi \neq 0} \frac{|h(\xi e^{\frac14 t})|}{|\xi e^{\frac14 t}|^k}
  = e^{-(1 - \frac14 k)t} \vertiii{h}_k.
\end{equation}
On the other hand,
\begin{equation*}
  |A(s,\xi)| \leq 2 \left|\psi\left(s, \frac{\xi}{2}\right)\right|,
\end{equation*}
since $\|\varphi(s,\cdot)\|_{L^\infty} \le \|g(s,\cdot)\|_{L^1}=1$ and 
$\|\bm{\Phi}\|_{L^\infty}  \le \|\bm{H}\|_{L^1}= 1$. This implies
\begin{equation}
  \label{eq:decay2}
   \vertiii{A(s)}_{k}
  \leq
  2 \sup_{\xi \neq 0} \frac{|\psi(s, \frac{\xi}{2})|}{|\xi|^k}
  =
  2^{1-k} \sup_{\xi \neq 0} \frac{|\psi(s, \frac{\xi}{2})|}{|\frac{\xi}{2}|^k}
  =
  2^{1-k}  \vertiii{\psi(s)}_{k}.
\end{equation}
Notice that $ \vertiii{\psi(t,\cdot)}_{k} < +\infty$ for all $0 \leq k < 3$,
since $\psi$ is a $\mathcal{C}^2$ function in $\xi$ with
$\psi(t,0) = \p_\xi \psi(t,0) = \p_\xi^2 \psi(t,0) = 0$. Using
\eqref{eq:decay1} and \eqref{eq:decay2} in \eqref{eq:decay-Duhamel} we
see that
\begin{equation*}\begin{split}
   \vertiii{\psi(t)}_k &\leq
\vertiii{T(t) \psi_0}_{k} + \int_0^t \vertiii{T(t-s) A(s)}_{k} \d s \\
&\leq
  \exp\left(-(1 - \frac14 k) t\right)\vertiii{\psi_0}_{k}
  +
  \int_0^t \exp\left(-(1 - \frac14 k) (t-s)\right) \vertiii{A(s)}_k \d s
  \\
  &\leq
  \exp\left(-(1 - \frac14 k) t\right)\vertiii{\psi_0}_{k}
  +
  2^{1-k} \int_0^t \exp\left(-(1 - \frac14 k) (t-s)\right) \vertiii{\psi(s)}_{k} \d s\,,
\end{split}\end{equation*}
which immediately gives by Gronwall's lemma that
\begin{equation*}
 \vertiii{\psi(t)}_{k} \leq \exp\left(-\sigma_{k} t\right) \vertiii{\psi_0}_{k}
  \qquad
  \text{with $\sigma_{k} := 1 - \frac14 k - 2^{1-k}$.}
\end{equation*}
We deduce the exponential convergence in Theorem \ref{k-norm-cvgce}. 

\noindent \textit{$\bullet$ The general case $p\geq1.$} For $1 \leq p < \infty$, we recall that the norms $\vertiii{\cdot}_{k,p}$, defined in \eqref{def:normkp}, are given by
\begin{equation*}
 \vertiii{\psi}_{k,p}^p := \ir \frac{|\psi(\xi)|^p}{|\xi|^{kp}} \d \xi,
\end{equation*}
and are well-defined if $|\psi(\xi)| \leq \min\{1 ,C|\xi|^3\}$ for some $C>0$  and $\frac{1}{p} < k < 3 + \frac{1}{p}$. With a similar calculation as before,
\begin{equation*}
  \vertiii{T(t) \psi}_{k,p}
  =\exp\left(-\alpha_p t\right) \vertiii{\psi}_{k,p}, \qquad  \alpha_p := 1 - \frac14 k + \frac{1}{4p}.
\end{equation*}
Also,
\begin{equation*}
  \vertiii{A(s)}_{k,p}
  \leq
  2^{1 + \frac{1}{p} - k} \vertiii{\psi(s)}_{k,p},
\end{equation*}
so we can repeat the same argument to obtain
\begin{equation*}\begin{split}
  \vertiii{\psi(t)}_{k,p} &\leq
  \vertiii{T(t) \psi_0}_{k,p} + \int_0^t \vertiii{T(t-s) A(s)}_{k,p} \d s \\
  &\leq
  \exp\left(-\alpha_p t\right)\vertiii{\psi_0}_{k,p}
  +
  \int_0^t \exp\left(-\alpha_p (t-s)\right) \vertiii{A(s)}_{k,p} \d s\\
& \leq
 \exp\left(-\alpha_p t\right)\vertiii{\psi_0}_{k,p}
  +
  2^{1 + \frac{1}{p} - k} \int_0^t \exp\left(-\alpha_p (t-s)\right) \vertiii{\psi(s)}_{k,p} \d s.
\end{split}\end{equation*}
 Then one concludes as previously using Gronwall's lemma. \end{proof}
 
\begin{rem}[\textit{\textbf{Invariance by scaling}}]\phantomsection\label{rem:scal}  
Theorem \ref{k-norm-cvgce} holds for solutions $g$ to \eqref{eq:IB-selfsim} satisfying the normalisation \eqref{eq:normalisation-g}. Recall that \eqref{eq:normalisation-g} is preserved by the nonlinear dynamics \eqref{eq:IB-selfsim}. We explain briefly how it applies to solutions of \eqref{eq:IB-selfsim} with positive energy (not necessarily unitary). Namely, assume that $\tilde{g}_{0}$ is an initial datum such that
$$\int_{\R}\tilde{g}_0(x)\dx=1, \qquad \int_{\R}\tilde{g}_{0}(x)\,x\dx=0, \qquad \int_{\R}\tilde{g}_{0}(x)x^2\dx=E >0$$
and let $\tilde{g}(t,x)$ be the associated solution to \eqref{eq:IB-selfsim}. Notice that $\tilde{g}(t,x)$ shares the same mass, momentum and energy with $\tilde{g}_0$ for any $t\geq0.$ Setting
$$g_{0}(x)= {\frac1{\lambda}}\,\tilde{g}_0\left( {\frac{x}{\lambda}} \right), \qquad \lambda= {\frac1{\sqrt{E}}},$$
one sees that $g_0$ satisfies \eqref{eq:normalisation-g}. Denoting by $g(t,x)$ the associated solution to \eqref{eq:IB-selfsim},  the scaling invariance property of $\Q_{0}$ implies that
$$g(t,x)= {\frac1{\lambda}}\,\tilde{g}\left(t, {\frac{x}{\lambda}} \right), \qquad \lambda= {\frac1{\sqrt{E}}}$$
while Theorem \ref{k-norm-cvgce} asserts that 
 \begin{equation*}
   \vertiii{\varphi(t) - \bm{\Phi}}_{k} \leq \exp\left(-\sigma_{k} t\right)\vertiii{\varphi_0 - \bm{\Phi}}_{k}\, \quad t \geq0,   \qquad \sigma_{k} := 1 - \frac14 k - 2^{1-k}\,,
  \end{equation*} 
 where $\varphi(t)$ is the Fourier transform of $g$ and $\bm{\Phi}$ that of $\bm{H}$. Denoting by  $\widetilde{\varphi}(t,\cdot)$ the Fourier transform of $\tilde{g}(t,\cdot)$, we have 
 $$\widetilde{\varphi}(t,\xi)=\varphi\left(t, {\frac{\xi}{\lambda}}\right) \qquad \text{ and } \qquad \widehat{H}_{ {\lambda}}(\xi)=\bm{\Phi}\left( {\frac{\xi}{\lambda}}\right),$$
where $\widehat{H}_{ {\lambda}}$ is the Fourier transform of the steady solution  
$$H_{ {\lambda}}(x)={ {\lambda}} \bm{H}\left({ {\lambda}} x\right), \qquad \lambda >0$$
of \eqref{eq:IB-selfsim} with unit mass, zero momentum and energy $E$. Since
$$\vertiii{\widetilde{\varphi}(t)-\widehat{H}_{ {\lambda}}}_{k}=\lambda^{k}\vertiii{\varphi(t)-\bm{\Phi}}_{k} \qquad \qquad \forall t\geq0$$
one sees that
\begin{equation*}
\vertiii{\widetilde{\varphi}(t) - \widehat{H}_{ {\lambda}}}_{k} \leq \exp\left(-\sigma_{k} t\right) \vertiii{\widetilde{\varphi}_0 - \widehat{H}_{ {\lambda}}}_{k}
    \qquad
    \text{with}  \qquad \sigma_{k} := 1 - \frac14 k - 2^{1-k}\,.
  \end{equation*}
In other words, for any choice of the initial energy $E >0,$ solutions to \eqref{eq:IB-selfsim} relax exponentially fast -- in the $\vertiii{\cdot}_{k}$ norm -- towards the unique steady solution with the prescribed energy $E$.
\end{rem}

 \section{Regularity estimates}\label{sec:regularity}
 
 The scope of this section is to study the propagation of regularity for the solutions to \eqref{eq:IB-selfsim}. which will result in the proof of Theorem \ref{theo:Sobolev}. As said, the propagation of regularity for \eqref{eq:IB-selfsim} has been addressed in \cite{FPTT} thanks to a semi-implicit discretization of the equation. We propose here a direct approach in which no iterative step is required. The strategy is purely based on comparison arguments and the construction of an upper barrier to solutions to \eqref{eq:IB-selfsim} in Fourier variable.

\subsection{Baseline regularity}\label{sec:baseline}

We begin our analysis by proving the propagation of baseline regularity of solutions, which in Fourier space follows by showing uniform propagation of decay at infinity.  To this end we present a series of lemmas with the main purpose of proving a comparison principle and showing a proper upper barrier for solutions of the rescaled Boltzmann model \eqref{eq:IB-selfsim}. 

The key argument consists in proving that estimates for \emph{low frequencies} transfer to \emph{large frequencies}. We start adopting the following notation for the drift term operator and its associated semigroup where we recall that we consider here the solutions to \eqref{eq:IB-selfsim} in Fourier variable \eqref{eq:selfsim-fourier}: we define the drift operator
\begin{equation}\label{ft:e3}
\mathcal{D} = \xi\,\partial_{\xi}
\end{equation}
and the operators
\begin{equation}\label{ft:e4}
\Gamma[u](\xi) = u\left(\frac{\xi}{2}\right)\,u\left(\frac{\xi}{2}\right)\,,\quad \text{and}\quad L\,u(\xi) = \,u\left(\frac{\xi}{2}\right)\,.
\end{equation}
With this notation, the Boltzmann equation in Fourier variable \eqref{eq:selfsim-fourier} reads
\begin{equation}\label{eq:FouriD}
\partial_{t}\varphi-\frac{1}{4}\mathcal{D}\varphi+\varphi= \Gamma[\varphi].
\end{equation}

\begin{lem}\phantomsection\label{lem:evol-family}
For a given bounded function $\sigma_{0} \in \mathcal{C}([0,\infty),L^{\infty}(\R))$, the unique solution $u\in \mathcal{C}(\Theta\,;\,L^{\infty}(\R))$ to 
\begin{equation}\label{eq:LS}
\partial_{t}u-\sigma_{0}(t,\cdot)Lu=0, \qquad u(s,s,\xi)=u_{0}, \qquad (s,t) \in  \Theta=\{(s,t) \in \R^{2}\;;t\geq s \geq0\}
\end{equation}is given by the following evolution family
$$u(s,t,\xi)=\mathcal{V}(s,t)u_{0}=\sum_{j=0}^{\infty}\mu_{j}(s,t,\xi)L^{j}u_{0}(\xi)=\sum_{j=0}^{\infty}\mu_{j}(s,t,\xi)u_{0}\left(\frac{\xi}{2^j}\right)$$
where $\mu_{0}(s,t,\xi)=1$ for any $s,t,\xi$ and
$$\mu_{j}(s,t,\xi)=\int_{\Delta_{t}^{j}(s)}\prod_{k=0}^{j-1}L^{k}\left(\sigma_{0}(s_{k},\cdot)\right)\d \bm{s}_{j}=\int_{\Delta_{t}^{j}(s)}\prod_{k=0}^{j-1}\sigma_{0}\left(s_{k},\frac{\xi}{2^{k}}\right)\d \bm{s}_{j}, \qquad j\geq1$$
with $\Delta_{t}^{j}(s)$ the simplex
$$\Delta_{t}^{j}(s)=\left\{\bm{s}_{j}=\left(s_{0},\ldots,s_{j-1}\right),\;\;\;s \leq s_{j-1} \leq s_{j-2}\leq \ldots \leq s_{1}\leq s_{0} \leq t\right\}$$
and $\d\bm{s}_{j}=\d s_{0}\ldots\d s_{j-1}$ is the usual Lebesgue measure on $\Delta_{t}^{j}(s)$.
\end{lem}
\begin{proof}The proof is by direct inspection. Write, for $t \geq s\geq0$,
$$v(s,t,\xi)=\sum_{j=0}^{\infty}\mu_{j}(s,t,\xi)L^{j}u_{0}(\xi).$$
Observe that $\mu_{0}(s,s,\xi)=1$, $\mu_{j}(s,s,\xi)=0$ for all $j \geq 1$ so that $v(s,s,\cdot)=u_{0}.$ On the one hand,
$$\partial_{t}v(s,t,\xi)=\sum_{j=1}^{\infty}\partial_{t}\mu_{j}(s,t,\xi)L^{j}u_{0}(\xi)=\sum_{j=0}^{\infty}\partial_{t}\mu_{j+1}(s,t,\xi)L^{j+1}u_{0}(\xi)$$
since we assumed $\mu_{0}=1$. On the other hand,
$$Lv(s,t,\xi)=\sum_{j=0}^{\infty}L\left(\mu_{j}(s,t,\xi)L^{j}u_{0}\right)=\sum_{j=0}^{\infty}L(\mu_{j}(s,t,\xi))L^{j+1}u_{0}(\xi)$$
since $L( w_{1}\,w_{2})=L(w_{1})L(w_{2})$ (if one of the  {$w_{i}$} is bounded at least for the product to make sense). Therefore, if 
$$\partial_{t}\mu_{j+1}(s,t,\cdot)=\sigma_{0}(t,\xi)L\mu_{j}(s,t,\cdot) \qquad \mu_{j+1}(s,s,\cdot)=0\qquad \qquad j\geq 0$$
one gets that $v(s,t,\xi)$ solves \eqref{eq:LS}. By induction, since $\mu_{0}\equiv 1$, one gets the desired expression for $\mu_{j}$, $j\geq1.$
\end{proof}
\begin{rem} If $\sigma_{0}$ is constant, say $\sigma_{0}(t,\xi)=\alpha$ and $s=0$, because the volume of the simplex $\Delta_{t}^{j}(s)=\Delta_{t}^{j}(0)$ is equal to $\frac{t^{j}}{j!}$ one gets 
$$u(t,\xi)=\sum_{j=0}^{\infty}\frac{\left(\alpha t\right)^{j}}{j!}L^{j}u_{0}$$ which is exactly the expression   of the semigroup generated by the bounded operator $\alpha L$. 
\end{rem}
Our key point for the analysis is the following comparison for sub- and super-solutions to \eqref{eq:selfsim-fourier}
 in the form \eqref{eq:FouriD}:

\begin{lem}[\textit{\textbf{Comparison lemma}}]\phantomsection \label{ft:l1}
Assume bounded continuous functions $u,v \ge 0$ satisfying
\begin{subequations}
\begin{align}
\partial_{t}u + \big(-\tfrac{1}{4}\mathcal{D} + 1\big)u&\geq\Gamma[u]\,,\label{comp:e1}\\
\partial_{t}v + \big(-\tfrac{1}{4}\mathcal{D} + 1\big)v&\leq\Gamma[v]\,,\label{comp:e2}
\end{align}
\end{subequations}
and $u(0,\cdot) \geq v(0,\cdot)$.  Then $u(t,\cdot)\geq v(t,\cdot)$ for any $t \geq 0$.
\end{lem}
\begin{proof}
For functions $u$ and $v$ satisfying \eqref{comp:e1}-\eqref{comp:e2} with 
$\|u(t)\|_{\infty} \leq M,$ $\|v(t)\|_{\infty} \leq M,$
define 
$$S(t,\xi):=u\left(t,\frac{\xi}{2}\right) + v\left(t,\frac{\xi}{2}\right) \in[0,2M].$$ Then, one deduces for the difference $h(t,\xi):=v(t,\xi)-u(t,\xi)$ the relation
\begin{equation*}
\partial_{t}h(t,\xi) -\frac{\xi}{4}\partial_{\xi}h(t,\xi) +h(t,\xi) \le v^2\left(t,\frac{\xi}{2}\right)-u^2\left(t,\frac{\xi}{2}\right)= S(t,\xi) \; h\left(t,\frac{\xi}{2}\right)\,.
\end{equation*}
Multiplying the above equation with  $\mathrm{sign}^+(h(t,\xi))$ and integrating with respect to $\xi\in \R$, we formally obtain 
$$\frac{\d}{\d t } \|h^+(t)\|_{L^1} +\frac54  \|h^+(t)\|_{L^1}\le \int_{\R}  S(t,\xi) \; h\left(t,\frac{\xi}{2}\right)\mathrm{sign}^+(h(t,\xi)) \, \d \xi.$$
Since $ S\ge 0$, it leads to 
$$\frac{\d}{\d t } \|h^+(t)\|_{L^1} +\frac54  \|h^+(t)\|_{L^1}\le \int_{\R}  S(t,\xi) \; h^+\left(t,\frac{\xi}{2}\right)\, \d \xi \le 2\|S(t)\|_{L^\infty} \|h^+(t)\|_{L^1}\,.$$
We then deduce from the Gronwall Lemma and $h^+(0)=0$ that $h^+(t)=0$ for any $t\ge 0$. Hence $v\le u $ in $(0,\infty)\times \R$. 
\end{proof}
With this, we have the following propagation of smoothness, which generalises \cite[Theorem 4]{FPTT}:
\begin{prp}[\textit{\textbf{Propagation of strong smoothness}}]\phantomsection\label{strong-smoothness} Let $\varphi(t)=\varphi(t,\xi)$ be a solution to the nonlinear equation \eqref{eq:selfsim-fourier}  with $\|\varphi(t)\|_{L^\infty}\leq 1$. Assume there exists $a >0$ such that $|\varphi(0,\xi)| \leq \bm{\Phi}(a\,\xi)$ for any $\xi \in \R$.  Then,
\begin{equation*}
|\varphi(t,\xi)|\leq\bm{\Phi}(a\,\xi) \quad \text{for all} \quad t\geq0\,,\,\,\xi \in \R.
\end{equation*}
\end{prp}
\begin{proof}
Set $v(t,\xi)=\left|\varphi\left(t,\frac{\xi}{a}\right)\right|$ for $a>0$ and $u(t,\xi)=\bm{\Phi}(\xi)$.  Since $v(0,\xi) = |\varphi(0,\frac{\xi}{a})| \leq \bm{\Phi}(\xi)$ and 
\begin{align*}
      \partial_t v(t,\xi)& =\frac{1}{2\left|\varphi\left(t,\frac{\xi}{a}\right)\right|} \left(\partial_t\varphi\left(t,\frac{\xi}{a}\right)\overline{\varphi\left(t,\frac{\xi}{a}\right)} + \varphi\left(t,\frac{\xi}{a}\right)\partial_t\overline{\varphi\left(t,\frac{\xi}{a}\right)}\right)\\
      & = \frac14 \xi \partial_\xi v(t,\xi) +  {\frac1{2\left|\varphi\left(t,\frac{\xi}{a}\right)\right|} \left(\varphi\left(t,\frac{\xi}{2a}\right)^2 \overline{\varphi\left(t,\frac{\xi}{a}\right)}+\overline{\varphi\left(t,\frac{\xi}{2a}\right)} ^2 \varphi\left(t,\frac{\xi}{a}\right)\right)} -v(t,\xi) 
    \end{align*}
    with 
$${\frac1{2\left|\varphi\left(t,\frac{\xi}{a}\right)\right|} \left(\varphi\left(t,\frac{\xi}{2a}\right)^2 \overline{\varphi\left(t,\frac{\xi}{a}\right)}+\overline{\varphi\left(t,\frac{\xi}{2a}\right)} ^2 \varphi\left(t,\frac{\xi}{a}\right)\right) \le  v\left(t,\frac{\xi}{2}\right)^2},$$
all conditions (inequalities \eqref{comp:e1} and \eqref{comp:e2} and initial condition) of Lemma \ref{ft:l1} are satisfied.  Therefore, $v(t,\xi)\leq u(t,\xi)$ or, equivalently, $|\varphi(t,\xi)|\leq\bm{\Phi}(a\,\xi)$ for all $t\geq0$.
\end{proof}
\begin{rem} The above result can be compared to \cite[Theorem 4]{FPTT} since $\bm{\Phi}$ is decaying exponentially fast for large $|\xi|$ (see \eqref{eq:Phi}). Interestingly, the result here is not associated to a physical counterpart $g(t,x)$ since the inverse Fourier transform of $\varphi$ may not be positive. 
\end{rem}

Now, we present two lemmas to relax in the above lemma the exponential decay on the initial data. For any $\beta \geq 0$, we set
$$\Psi_{\beta}(r)=\langle r\rangle^{-\beta}=\left(1+r^{2}\right)^{-\frac{\beta}{2}}, \qquad r >0.$$
We will use repeatedly that $\Psi_{\beta}(\cdot)$ is non increasing with moreover
$$\Psi_{\beta}(r) \leq \min\left(1,r^{-\beta}\right) \qquad \forall r >0.$$
We have the following \emph{short time estimate} for sub-solutions to \eqref{eq:FouriD}

\begin{lem}[\textit{\textbf{Short time estimate}}]\phantomsection\label{short-time-lemma}
Fix $\beta>0$.  Assume $u(t,\xi)\in[0,1]$ satisfies the inequality
\begin{equation}\label{ft:e12}
\partial_{t}u + \big(-\tfrac{1}{4}\mathcal{D} + 1\big)u \leq \Gamma[u] \,
\end{equation}
together with
$$0\leq u(0,\xi)=u_{0}(\xi) \leq \Psi_{\beta}(|\xi|) \quad \forall \xi \in \R.$$
Assume there is $\delta >0$ such that
$$u(t,\xi)\leq\Psi_{\beta}(|\xi|) \qquad \quad \text{ for } \quad |\xi|\leq\delta, \qquad t \geq 0.$$
Then, for any $\beta' \in \left(0,\frac{\beta}{2}\right]$, there exists $\tau(\delta, \beta, \beta')>0$ such that 
$$0\leq u(t,\xi)\leq\Psi_{\beta'}(|\xi|) \qquad \text{ for any } t\in[0,\tau(\delta,\beta,\beta')], \qquad \xi \in \R.$$  The time $\tau(\delta,\beta,\beta')$ satisfies $\lim_{\beta'\rightarrow0}\tau(\delta,\beta,\beta')=+\infty$ for any fixed $\delta>0$ and $\beta>0$.
\end{lem}
\begin{proof} Let $\mathcal{U}(t)$ be the semigroup associated to the generator $-\frac{1}{4}\mathcal{D}$, i.e. $\mathcal{U}(t)f(\xi)=f(\xi e^{-\frac{1}{4}t})$. Setting $w(t,\xi)=e^{t}\,\mathcal{U}(t)u(t,\xi)$ we write \eqref{ft:e12} as 
$$\partial_{t}u(t,\xi) + \big(-\tfrac{1}{4}\D +1\big)u(t,\xi) \leq u\Bigl(t,\frac{\xi}{2}\Bigr)Lu(t,\xi)$$ or equivalently
$$\partial_{t}w(t,\xi)\leq \mathcal{U}(t)u\Bigl(t,\frac{\xi}{2}\Bigr)Lw(t,\xi)$$ 
and denote by $\left(\mathcal{V}(s,t)\right)_{s,t}$ the evolution family constructed in Lemma \ref{lem:evol-family} with the choice
$$\sigma_{0}(t,\xi)=\left[\mathcal{U}(t)u\Bigl(t,\frac{\xi}{2}\Bigr)\right]=u\left(t,\frac{\xi}{2}e^{-\frac{1}{4}t}\right)\,.$$
{One has $\partial_{t}w(t,\xi) -\sigma_{0}(t,\xi)Lw(t,\xi) \leq 0$ whereas $\varpi(t,\xi)=\mathcal{V}(0,t)u_{0}(\xi)$ is a solution to $\partial_{t}\varpi(t,\xi)-\sigma_{0}(t,\xi)L\varpi(t,\xi)=0$. Therefore, arguing as in Lemma \ref{ft:l1},}
$$0 \leq w(t,\xi) \leq \mathcal{V}(0,t)u_{0}(\xi), \qquad \text{ with }  \quad \mathcal{V}(0,t)u_{0}(\xi)=\sum_{j=0}^{\infty}\nu_{j}(t,\xi)L^{j}u_{0}(\xi)$$
where $\nu_0(t,\xi)=1$ and 
$$\nu_{j}(t,\xi)=\int_{0}^{t}u\left(s_{0},\frac{\xi}{2}e^{-\frac{1}{4}s_{0}}\right)\nu_{j-1}\left(s_{0},\frac{\xi}{2}\right)\d s_{0}, \qquad  t \geq 0,\xi \in \R.$$
Then
\begin{equation}\label{eq:semi}
0 \leq u(t,\xi) \leq e^{-t}\sum_{j=0}^{\infty}\nu_{j}(t,\xi\,e^{\frac{1}{4}t})L^{j}u_0\left(\xi\,e^{\frac{1}{4}t}\right).\end{equation}
Since $u_{0}(\xi) \leq \Psi_{\beta}(|\xi|)$, 
$$0 \leq u(t,\xi) \leq e^{-t}\sum_{j=0}^{\infty}\nu_{j}(t,\xi\,e^{\frac{1}{4}t})L^{j}\Psi_{\beta}\left(|\xi|\,e^{\frac{1}{4}t}\right)=e^{-t}\sum_{j=0}^{\infty}\nu_{j}(t,\xi\,e^{\frac{1}{4}t})\Psi_{\beta}\left(2^{-j}|\xi|\,e^{\frac{1}{4}t}\right).$$
In addition, since $\Psi_{\beta}$ is non increasing and $2^{-j}|\xi|\,e^{\frac{1}{4}t}\geq 2^{-j}|\xi|$, it holds that
$$
0 \leq u(t,\xi) \leq e^{-t}\sum_{j=0}^{\infty}\nu_{j}(t,\xi\,e^{\frac{1}{4}t})\Psi_{\beta}\left(2^{-j}|\xi|\right).$$ 
By assumption $u(t,\xi)\in [0,1]$, therefore
\begin{equation}\label{eq:nuj}
\nu_{j}(t,\xi\,e^{\frac{1}{4}t}) \leq  \frac{t^{j}}{j!} \qquad \text{ and } \qquad  0 \leq u(t,\xi) \leq e^{-t}\sum_{j=0}^{\infty}\frac{t^{j}}{j!}\Psi_{\beta}\left(2^{-j}|\xi|\right).
\end{equation}
Observe that $\langle r \rangle^{a} \geq \langle \sqrt{a}\,r \rangle$ for any $a\geq1$  and $r\ge 0$ so that, for any $\beta\geq 2\beta'$, 
$$\Psi_{\beta}(|\xi|)\leq \Psi_{2\beta'}\left(\sqrt{\frac{\beta}{2\beta'}}\,|\xi|\right)$$ for any $\xi \in\R.$  Consequently,
$$
\Psi_{\beta}\big(2^{-j}\,|\xi|\big)\leq 2^{2\beta'j}\Psi_{2\beta'}\left(\sqrt{\frac{\beta}{2\beta'}}\,|\xi|\right)\leq 2^{2\beta'j}\Psi_{\beta'}\left(\sqrt{\frac{\beta}{2\beta'}}\,\delta\right)\Psi_{\beta'}(|\xi|), \qquad |\xi|\geq\delta\,.$$
Using this estimate in inequality \eqref{eq:nuj}, one sees that, for any $|\xi| \geq \delta$, it holds
\begin{equation}\label{ft:e14}
u(t,\xi)\leq  { e^{-t}}\Psi_{\beta'}(|\xi|)\Psi_{\beta'}\left(\sqrt{\frac{\beta}{2\beta'}}\,\delta\right)\sum_{j=0}^{\infty}\frac{t^{j}}{j!}2^{2\beta'\,j}\
=\Psi_{\beta'}(|\xi|)\Psi_{\beta'}\left(\sqrt{\frac{\beta}{2\beta'}}\,\delta\right)e^{(2^{2\beta'}-1)t}\,.
\end{equation}
Thus, choosing
\begin{equation*}
\tau(\delta,\beta,\beta')=\frac{ \beta' \ln\big(\langle \sqrt{\frac{\beta}{2\beta'}}\delta \rangle\big)}{2^{2\beta'}-1}\,,
\end{equation*}
we have 
$$u(t,\xi)\leq\Psi_{\beta'}(|\xi|) \qquad \text{ for } |\xi|\geq\delta \quad \text{ and } \quad t\in[0,\tau(\delta,\beta,\beta')].$$ Since, by assumption, for $|\xi| \leq \delta$ it holds $u(t,\xi)\leq\Psi_{\beta}(\xi)\leq \Psi_{\beta'}(|\xi|)$ we deduce that 
$$u(t,\xi)\leq\Psi_{\beta'}(|\xi|)$$
holds true for \emph{any} $\xi \in \R$ and $t \in [0,\tau(\delta,\beta,\beta')]$. From the definition of $\tau$, it  is clear that $\lim_{ { \beta'\rightarrow 0}} \tau(\delta,\beta,\beta')=+\infty$. 
\end{proof}

\begin{lem}[\textit{\textbf{Global-in-time estimates}}]\phantomsection\label{Long-time} Assume $u(t,\xi)\in[0,1]$ satisfies the inequality \eqref{ft:e12} for any $t\geq0$ with
$u(0,\xi)=u_0(\xi)\leq \Psi_{\beta}(|\xi|)$ for any $\xi \in \R$, for some $\beta>0$.  If 
$$u(t,\xi) \leq \Psi_{\beta}(|\xi|) \qquad \text{ for } |\xi| \leq 4, \qquad t \geq0\,,$$ 
then $u(t,\xi) \leq \Psi_{\beta}(|\xi|)$ for all $\xi\in\mathbb{R}$, $t\ge 0$.
\end{lem}
\begin{proof} Inequality \eqref{ft:e12} together with
Duhamel's formula gives that
\begin{equation*}
u(t,\xi)\leq  u_{0}\left(\xi\,e^{\frac{1}{4}t}\right)e^{-t} +\int^{t}_{0}e^{-(t-s)}\left[u\left(s,\frac{\xi}{2}e^{\frac{1}{4}(t-s)}\right)\right]^{2}\d s, \qquad t \geq0.
\end{equation*}
For a given $t \geq0$, recall that $u_{0}\left(\xi\,e^{\frac{1}{4}t}\right) \leq \Psi_{\beta}\left(|\xi|\,e^{\frac{1}{4}t}\right)\leq \Psi_{\beta}(|\xi|)$ whereas, if
$|\xi| \leq 8e^{-\frac{t}{4}}$ then $\frac{|\xi|}{2}e^{\frac{1}{4}(t-s)} \leq 4$ for all $s \in [0,t]$ which by assumption gives 
$$ u\left(s,\frac{\xi}{2}e^{\frac{1}{4}(t-s)}\right)  \leq \Psi_{\beta}\left(\frac{|\xi|}{2}e^{\frac{1}{4}(t-s)}\right) \leq \Psi_{\beta}\left(\frac{|\xi|}{2}\right) \qquad \forall s \in [0,t]$$
where we used that $\Psi_{\beta}(\cdot)$ is non increasing. Consequently
$$u(t,\xi) \leq \Psi_{\beta}(|\xi|)e^{-t} + \Psi_{\beta}\left(\frac{|\xi|}{2}\right)^{2}(1-e^{-t})\,,\quad 0\leq |\xi|\leq 8e^{-t/4}\,.$$
In particular, setting 
$$t_{0}:=4\log\frac{4}{3}$$
so that $|\xi| \leq 6 \Longrightarrow |\xi| \leq 8e^{-\frac{t}{4}}$ for $t \in [0,t_{0}]$, one deduces that
\begin{equation}\label{eq:upsib}
u(t,\xi) \leq \Psi_{\beta}(|\xi|)e^{-t} + \Psi_{\beta}\left(\frac{|\xi|}{2}\right)^{2}(1-e^{-t})\,,\quad 0\leq |\xi|\leq 6, \qquad t \in [0,t_{0}].\end{equation}
Since $\Psi_{\beta}\left(\frac{|\xi|}{2}\right)^{2}\leq \Psi_{\beta}(|\xi|)$ for $|\xi|\geq {\sqrt{8}}$, one deduces that,
 \begin{equation*}
u(t,\xi) \leq  \Psi_{\beta}(|\xi|)\,,\qquad \qquad  {\sqrt{8}} \leq |\xi|\leq 6\,\qquad t \in [0,t_{0}]
\end{equation*}
which, by assumption, yields
$$u(t,\xi) \leq \Psi_{\beta}(|\xi|) \qquad \quad \text{ for all } 0 \leq |\xi| \leq 6, \qquad t \in [0,t_{0}].$$
Iterating this process $k$-times one gets 
\begin{equation*}
u(t,\xi) \leq \Psi_{\beta}(|\xi|)\,,\qquad 0\leq |\xi|\leq 4\cdot \left(\frac{3}{2}\right)^{k}\,,\quad t\in[0,t_{0}]\,.
\end{equation*}
Since $k$ is arbitrary, we get
$$u(t,\xi) \leq \Psi_{\beta}(|\xi|), \qquad \text{ for all } \xi \in \R\,,\quad t \in [0,t_{0}].$$ Since then, for any $t \geq t_{0}$
$$u(t,\xi) \leq e^{-(t-t_{0})}u\left(t_0,\xi\,e^{\frac{1}{4}(t-t_{0})}\right)+\int_{0}^{t-t_{0}}e^{-(t-t_{0}-s)}\left[u\left(s+t_0,\frac{\xi}{2}e^{\frac{1}{4}(t-t_{0}-s)}\right)\right]^{2}\d s$$
one can reproduce the above argument to show that the bound $u(t,\xi) \leq \Psi_{\beta}(|\xi|)$ holds also on the interval $[t_{0},2t_{0}]$. Iterating the procedure, the bound holds for any time $t\geq0$ and any $\xi \in \R.$
\end{proof}
We are in conditions to prove the main result of the section.
\begin{theo}\label{theo-baseline}
Let $\varphi(t,\xi)$ be a solution of the self-similar problem \eqref{eq:selfsim-fourier} satisfying $\|\varphi(t)\|_{L^\infty}\leq 1$ and with initial condition $\varphi_0$ enjoying the regularity
\begin{equation}\label{eq:assPsi}
\vertiii{\varphi_0 - \bm{\Phi}}_{k} <\infty \qquad \text{and} \qquad |\varphi_{0}(\xi)|\leq \Psi_{\beta}(c|\xi|), \qquad \forall \xi \in\R
\end{equation}
for some $k\in(2,3)$, $c\in(0,1]$, and $\beta>0$.  Then,
\begin{equation}\label{under}
\sup_{t\geq0}|\varphi(t,\xi)| \leq \Psi_{\beta}(c_{0}|\xi|)\,,\,\qquad \quad \xi \in \R
\end{equation}for some positive constant $c_{0} >0$ depending only on $\beta$, $c$, and $\vertiii{\varphi_0 - \bm{\Phi}}_{k}$.
\end{theo}
\begin{proof}
Let $k \in (2,3)$ be given with $C_{k}:=\vertiii{\varphi_{0}-\bm{\Phi}}_{k} < \infty.$ Theorem \ref{k-norm-cvgce} states that
\begin{equation*}
\vertiii{\varphi(t) - \bm{\Phi}}_{k} \leq C_{k}\,\exp\left(-\sigma_k t\right), \qquad \forall t \geq 0\end{equation*}
 with $\sigma_k = 1 - \frac14 k - 2^{1-k}>0\,.$ Therefore, for any $\xi \in \R$ 
\begin{equation}\label{eq:varPhiCk} 
|\varphi(t,\xi)| \leq \bm{\Phi}(\xi) + C_{k}|\xi|^{k}\exp\left(-\sigma_k t\right) \leq (1+|\xi|)e^{-|\xi|}+C_{k}|\xi|^{k}
\,\qquad \forall  t \geq0.\end{equation}
For any $\alpha \in (0,1),$ the mapping $F(r)=(1+r)e^{-r}+C_{k}r^{k}-(1+r^{2})^{-\frac{\alpha}{2}}$ is such that
$$F(0)=F'(0)=0, \qquad F''(0)=-1+\alpha < 0$$
from which one sees that there is $\delta >0$ (depending on $\alpha$ and $C_{k}$) such that $F(r) \leq 0$ for $r \in (0,\delta)$, i.e. 
\begin{equation}\label{small-freq}
|\varphi(t,\xi)| \leq \Psi_{\alpha}(|\xi|) \qquad \forall |\xi| \leq \delta\,,\,\quad t\geq0.\end{equation}
For large time, we  introduce, for $\alpha \in (0,1)$,
$$G_{t}(r)=(1+r)e^{-r}+C_{k}r^{k}e^{-\sigma_k t}-(1+r^{2})^{-\frac{\alpha}{2}}, \qquad r >0.$$
One first observes that
\begin{equation}\label{pou}
G_{t}(r) \leq (1+r)e^{-r}-1+\frac{\alpha}{2}r^{2} + C_{k}r^{k}=(1+r) e^{-r} -1+ \alpha r^2 + C_k r^k - \frac{\alpha}2 r^2,
\end{equation}
with  
$$(1+r) e^{-r} -1+ \alpha r^2 \le 0 \qquad  \mbox{ for any } 0 \le r \le 4,$$
 when $\alpha < \frac{e^{-4}}2$ and  $$ C_k r^k - \frac{\alpha}2 r^2 \le 0 \qquad \mbox{ for any }  0 \le r \le \left(\frac{\alpha}{2C_k}\right)^{\frac{1}{k-2}}.$$ 
 Therefore,  if $\alpha < \frac{e^{-4}}2$, then
$$ G_{t}(r) \le 0 \qquad \mbox{ for any } t\ge0 \mbox{ and } 0 \le r \le r_{\alpha,k},$$ 
where  $r_{\alpha,k}:=\min\left\{ \left(\frac{\alpha}{2C_k}\right)^{\frac{1}{k-2}},4\right\}>0$. Now, for $ r_{\alpha,k}\le r\le 4$, we have, again with \eqref{pou} 
$$G_t(r)\le h_{\alpha} (r_{\alpha,k}) + C_{k} 4^{k}e^{-\sigma_k t},$$
since $h_{\alpha}(r):=(1+r)e^{-r}-1+\frac{\alpha}{2}r^{2}$ is decreasing on $[r_{\alpha,k},4]$ when $\alpha < \frac{e^{-4}}2  < e^{-4}$. Note that $h_{\alpha}(r_{\alpha,k})<0$.  Choosing $t_*\ge \frac{-1}{\sigma_k} \log\left(-\frac{1}{C_k 4^k} h_{\alpha}(r_{\alpha,k})\right)$, we obtain that 
$$\max_{0 \leq r \leq 4}G_{t}(r) \leq 0, \qquad \forall t \geq t_{*}.$$
From this we  conclude that
\begin{align}\label{large-freq}
|\varphi(t,\xi)|\leq\Psi_{\alpha}(|\xi|)\,,\quad\text{for}\quad |\xi| \leq 4\,,\quad t \geq t_{*}\,.
\end{align}
Observe now that $\Psi_{\beta}(c\,|\xi|)\leq\Psi_{\alpha}\big(\sqrt{\frac{\beta}{\alpha}}\,c \,|\xi|\big)$ for $\alpha \in {(0, \,\beta]}$.  Hence, choosing $\alpha=\min\{ {\frac{e^{-4}}{2}},\,c^2\beta\}$ it holds $|\varphi_{0}(\xi)|\leq \Psi_{\alpha}(|\xi|)$ for any $\xi \in \R$. Given the estimate \eqref{small-freq}, we may invoke Lemma \ref{short-time-lemma} with $u(t,\xi)=|\varphi(t,\xi)|$, $\alpha\in (0,1)$, and $\alpha'\in(0, {\alpha/2}]$ sufficiently small such that $\tau(\delta,\alpha,\alpha')\geq t_{*}$ to obtain that
\begin{equation*}
|\varphi(t,\xi)| \leq \Psi_{\alpha'}(|\xi|)\,,\qquad \xi\in\mathbb{R}\,,\quad t\in[0,t_{*}]\,.
\end{equation*}
With this and the estimate \eqref{large-freq} we use Lemma \ref{Long-time} in the interval $[t_{*},\infty)$, with $u(t,\xi)=|\varphi(t,\xi)|$ and $ {\beta}=\alpha'$, to conclude that 
\begin{equation}\label{eq:Psib'}
|\varphi(t,\xi)|\leq \Psi_{\alpha'}(|\xi|) \qquad \text{ for all } \xi \in \R\,,\quad t\geq0\,.
\end{equation}
In order to upgrade the decay rate up to $\beta$, we can bootstrap the previous estimate after noticing that, thanks to \eqref{eq:Psib'}, 
\begin{equation*}
\Big| \varphi\Big(t, \frac\xi2 \Big)^{2} \Big| \leq \Psi_{2\alpha'}\left(\frac{|\xi|}{2}\right)\,
\end{equation*}
so that, $u(t,\xi)=|\varphi(t,\xi)|$ satisfies $\partial_{t}u + \big(-\tfrac{1}{4}\mathcal{D} + 1\big)u \leq \Psi_{2\alpha'}\left(\frac{|\xi|}{2}\right)$. Using Duhamel's formula, it holds that
\begin{equation*}
|\varphi(t,\xi)| \leq \max\left\{\Psi_{\beta}(c\,|\xi|),\Psi_{2\alpha'}\left(\frac{|\xi|}{2}\right)\right\}\,,\qquad \qquad t \geq0.
\end{equation*} 
Iterating this process, we see that, for any $j \in \N$, $j \geq 1$,
$$|\varphi(t,\xi)| \leq \max\left\{\Psi_{\beta}(c\,|\xi|),\Psi_{2\beta}\left(\frac{c\,|\xi|}{2}\right),\ldots,\Psi_{2^{j-1}\beta}\left(\frac{c\,|\xi|}{2^{j-1}}\right),\Psi_{2^{j}\alpha'}\left(\frac{|\xi|}{2^{j}}\right)\right\}\,$$
holds for any $\xi\in \R$ and $t \geq 0.$ Notice that
\begin{multline*}
  \max\left\{\Psi_{\beta}(c\,|\xi|),\Psi_{2\beta}\left(\frac{c\,|\xi|}{2}\right),\ldots,\Psi_{2^{j-1}\beta}\left(\frac{c\,|\xi|}{2^{j-1}}\right),\Psi_{2^{j}\alpha'}\left(\frac{|\xi|}{2^{j}}\right)\right\} \\
  \leq \max\left\{ \Psi_{\beta}\left(\frac{c\,|\xi|}{2^{j-1}}\right), \Psi_{2^{j}\alpha'}\left(\frac{|\xi|}{2^{j}}\right)\right\} \leq \Psi_{\beta}\left(\frac{c\,|\xi|}{2^{j}}\right),
  \end{multline*}
as soon as $2^{j}\alpha' \geq \beta$. Setting
$$c_{0}=c \,2^{-j}\quad\text{with}\quad j=\Big\lfloor\frac{\log\big(\beta/\alpha'\big)}{\log 2}\Big\rfloor+1$$
the above condition is satisfied  and the result proved.
\end{proof}

\begin{rem} As previously, our result is not associated to a physical counterpart $g(t,x)$, yet it requires boundedness $\|\varphi(t)\|_{L^\infty}\leq1$ linked to the mass of $g(t,x)$. 
We observe that, as pointed out in \cite[Lemma 14]{FPTT}, if a function $0 \leq h\in L^{1}$ with unitary norm satisfies that $\sqrt{h}\in \dot{H}^{\alpha}(\R)$ then $|\widehat{h}(\xi)| \leq \Psi_{\alpha}(c\,|\xi|)$ with $c^{-\alpha}=\max\{2,2^{\alpha}\}\|\sqrt{h}\|_{\dot{H}^{\alpha}}$.
\end{rem}

\begin{rem} Notice that, Theorem \ref{theo-baseline} provides the uniform in time propagation of regularity for the solution to \eqref{eq:IB-selfsim} under the strong assumption \eqref{eq:assPsi}. Indeed, observing that
$$\int_{\R}\left|\Psi_{\alpha}(c|\xi|)\right|^{2}(1+|\xi|^{2})^{s}\d \xi < \infty \qquad \forall 0 \leq s < \alpha-\frac{1}{2}$$
we deduce that, under assumption \eqref{eq:assPsi}, a solution $g=g(t,x)$ to \eqref{eq:IB-selfsim} belongs to $H^{s}(\R)$ for any $0\le s< \alpha-\frac{1}{2}$ with a uniform in time estimate on $\|g(t)\|_{H^{s}}$.\end{rem}
\subsection{Higher regularity norms -- Proof of Theorem~\ref{theo:Sobolev}} As just observed, Theorem \ref{theo-baseline} provides the uniform in time  regularity $H^{s}(\R)$ bounds for the solution to \eqref{eq:IB-selfsim} for small values of $s.$ We aim now to prove the regularity  as well as the convergence in Sobolev spaces $H^{\ell}(\R)$ with higher regularity. The starting point is the propagation of baseline regularity \eqref{under}:
\begin{proof}[Proof of Theorem~\ref{theo:Sobolev}] As before, we call $\varphi = \varphi(t,\xi)$
    the Fourier transform of $g$ which is a solution to
\eqref{eq:selfsim-fourier} with the normalisation
\eqref{eq:normalisation-varphi}. The assumption on $\widehat{g}_{0}(\xi)$ means that
  $$|\varphi_{0}(\xi)| \leq \Psi_{\beta}(c|\xi|) \qquad \forall \xi \in\R$$
  for some $c \in (0,1)$ and $\beta >0$. From \eqref{under} we deduce that
  $$|\varphi(t,x)| \leq \Psi_{\beta}(c_{0}|\xi|), \qquad \forall t \geq0, \,\xi \in \R$$
  for some positive constant $c_{0}=c_{0}(\beta,c,\vertiii{\varphi_{0}-\bm{\Phi}}_{k}) >0$ provided $\vertiii{\varphi_{0}-\bm{\Phi}}_{k} < \infty$ where $\bm{\Phi}$ is given by \eqref{eq:Phi}. We consider the difference $\psi(t,\xi)
:= \varphi(t,\xi) - \bm{\Phi}(\xi)$  which satisfies \eqref{nlinear-self-similar}. We introduce the notation 
$$\phi_{m}:=|\xi|^{m}\phi$$ for any $m>0$ and any mapping $\phi=\phi(\xi)$.  Multiplying the self-similar equation \eqref{nlinear-self-similar} by $|\xi|^{m}$ we obtain that the mapping $\psi_{m}(t,\xi)=|\xi|^{m}\psi(t,\xi)$ satisfies
\begin{equation*}
  \p_t \psi_m =
  \frac14 \,\xi\,  \p_\xi \psi_m
  + 2^{m}\psi_m\Big(t, \frac{\xi}{2} \Big)\varphi\Big(t, \frac{\xi}{2}\Big)
    + 2^{m}\psi\Big(t, \frac{\xi}{2} \Big)\bm{\Phi}_{m}\Big(\frac{\xi}{2}
    \Big) - \Big(1+\frac{m}{4}\Big)\psi_m.
\end{equation*}
We define $\left(T_{m}(t)\right)_{t\geq0}$ the semigroup associated to $\frac14 \xi  \p_\xi - \big(1+\frac{m}{4}\big)$, i.e.
$$T_{m}(t)g(\xi)=e^{-\left(1+\frac{m}{4}\right)t}g\left(\xi\,e^{\frac{1}{4}t}\right), \qquad t \geq0$$
and
\begin{equation*}
A_{m}(t,\xi) := 2^{m}\psi_m\Big(t, \frac{\xi}{2} \Big)\varphi\Big(t, \frac{\xi}{2}\Big)\,,\qquad B_{m}(t,\xi):=2^{m}\psi\Big(t, \frac{\xi}{2} \Big)\bm{\Phi}_{m}\Big( \frac{\xi}{2}
    \Big)
\end{equation*}
so that
\begin{equation} \label{eq:decay-Duhamel-m}
  \psi_{m}(t) = T_{m}(t) \psi_{m}(0) + \int_0^t T_{m}(t-s)\big( A_{m}(s) + B_{m}(s) \big)  \d s\,.
\end{equation}
Note that, arguing as in the proof of Theorem~\ref{k-norm-cvgce}, for any suitable $h$,
\begin{equation*}
  \vertiii{T_{m}(t) h }_{k,p}
  = \exp\left(-\alpha_{m,p} t\right) \vertiii{h}_{k,p} \quad \text{ with } \quad  \alpha_{m,p} := 1 + \frac{m-k}{4} + \frac{1}{4p}\,,
\end{equation*}
while, for any $s \geq0, m \geq \beta$,
\begin{equation*}\begin{split} 
  \vertiii{A_{m}(s)}_{k,p}
   \leq
  2^{m - k + \frac{1}{p}} \vertiii{\psi_{m}(s)\,\varphi(s)}_{k,p}&\leq 2^{m - k + \frac{1}{p}}\vertiii{\psi_{m-\beta}(s)}_{k,p}\|\varphi_{\beta}(s)\|_{L^\infty},\\
 \vertiii{B_{m}(s)}_{k,p}
  &\leq
  2^{m - k + \frac{1}{p}}\| \bm{\Phi}_m\|_{L^\infty} \vertiii{\psi(s)}_{k,p}\,.
\end{split}\end{equation*}
Observe that H\"older's inequality implies that
\begin{equation*}
\vertiii{\psi_{m-\beta}(s)}_{k,p} \leq \vertiii{\psi_{m}(s)}^{1-\frac{\beta}{m}}_{k,p}\vertiii{\psi(s)}^{\frac{\beta}{m}}_{k,p}\,,\qquad m\geq\beta, 
\end{equation*}
and \eqref{under} leads to $\|\varphi_{\beta}(s)\|_{ L^\infty}\leq c^{-\beta}_{0}$.  Consequently, using Young's inequality we are led to
\begin{equation*}\begin{split}
  \vertiii{\psi_{m}(t)}_{k,p} &\leq
  \vertiii{T_{m}(t)\psi_{m}(0)}_{k,p} + \int_0^t \vertiii{T_{m}(t-s)\big(A_{m}(s) + B_{m}(s)\big)}_{k,p} \d s
  \\
  &\leq
  e^{-\alpha_{m,p} t}\vertiii{\psi_{m}(0)}_{k,p}
  +
  \int_0^t e^{-\alpha_{m,p} (t-s)} \big( \vertiii{ A_{m}(s) }_{k,p} + \vertiii{ B_{m}(s) }_{k,p}\big) \d s
  \\
  &\leq
  e^{-\alpha_{m,p} t}\vertiii{\psi_{m}(0)}_{k,p}
  +
   \int_0^t e^{-\alpha_{m,p} (t-s)} \Big(\varepsilon\,\vertiii{\psi_{m}(s)}_{k,p} 
  + \frac{C}{\varepsilon^{\frac{m}{\beta}-1}} \vertiii{\psi(s) }_{k,p}\Big)\d s,
\end{split}\end{equation*}
for a constant that can be taken as $C:=c^{-m}_{0}\,2^{(m-k+\frac1p)\frac m\beta} + 2^{m - k + \frac{1}{p}} \, \| \bm{\Phi}_m\|_{L^\infty}$.  Note that 
$$\alpha_{m,p}>\sigma_{k}(p),$$  
where we recall that $\sigma_{k}(p) = 1 - \frac14 k + \frac{1}{4p} - 2^{1 + \frac{1}{p} - k}$.  Therefore, thanks to Theorem \ref{k-norm-cvgce} it follows that
\begin{equation*}
\int_0^t e^{-\alpha_{m,p} (t-s)}\vertiii{\psi(s) }_{k,p}\d s \leq \frac{e^{-\sigma_{k}(p) t} }{\alpha_{m,p}-\sigma_{k}(p)}\, \vertiii{\psi(0)}_{k,p} \,\qquad \qquad t\geq0.
\end{equation*} 
As a consequence, calling $u(t) := e^{\sigma_{k}(p) t}\,\vertiii{ \psi_{m}(t)}_{k,p}$ we see that
\begin{equation*}
u(t) \leq \vertiii{\psi_{m}(0)}_{k,p} +  \frac{C\, \vertiii{\psi(0) }_{k,p}}{\varepsilon^{\frac{m}{\beta}-1}(\alpha_{m,p}-\sigma_{k}(p))} + \varepsilon \int_0^t u(s) \d s,
\end{equation*}
which, by Gronwall's lemma, immediately gives  that
\begin{equation}\label{higher-norm:relax}
  \vertiii{\psi_{m}(t)}_{k,p} \leq e^{-(\sigma_{k}(p) - \varepsilon) t}\,\bigg(\vertiii{\psi_{m}(0)}_{k,p} +  \frac{C\, {\vertiii{\psi(0)}_{k,p}}}{\varepsilon^{\frac{m}{\beta}-1}(\alpha_{m,p}-\sigma_{k}(p))}\bigg)\,.
\end{equation}
One chooses  $p=2$ and $m = k$ so that 
\begin{equation*}
\vertiii{\psi_{m}(t)}_{k,p} = \| \psi(t) \|_{L^2} = \| g(t) - \bm{H} \|_{L^2}\,
\end{equation*} 
thanks to Parseval identity. Moreover, one has,  {for $2<k<3$ (see Lemma \ref{lem:mukvertk})}
\begin{equation*}
 {\vertiii{\psi(0)}_{k,2}} \leq C\| g_0 - \bm{H}\|_{L^1(\w_{k})}\,.
\end{equation*}
Consequently, from \eqref{higher-norm:relax} one obtains the exponential relaxation in $L^{2}(\mathbb{R})$ as
\begin{equation}\label{L2:relax}
  \|g(t)- \bm{H}\|_{L^2} \leq e^{-(\sigma_{k}(2) - \varepsilon) t}\,\bigg( \| g_0 - \bm{H}\|_{L^2} +  \frac{C\,\| g_0 - \bm{H} \|_{ L^{1}(\w_k) } }{\varepsilon^{\frac{k}{\beta} - 1}(\alpha_{k,2}-\sigma_k(2))}\bigg)\,,\qquad \frac52 < k < {3}.
\end{equation}
More generally, for any $\ell>0$ one can choose $m=\ell+k$, $p=2$ and use the fact that 
\begin{equation*}
\vertiii{\psi_{\ell+k}(t)}_{k,2} = \| |\cdot|^{\ell}\psi(t) \|_{L^2} = \| g(t) - \bm{H} \|_{\dot{H}^{\ell}}\,.
\end{equation*} 
Consequently, \eqref{higher-norm:relax} implies that
\begin{equation}\label{Hom-H-relax}
  \| g(t) - \bm{H}\|_{\dot{H}^{\ell}} \leq e^{-(\sigma_{k}(2) - \varepsilon) t}\,\bigg( \| g_0 - \bm{H}\|_{\dot{H}^{\ell}} +  \frac{C\,\| g_0 - \bm{H} \|_{ L^{1}(\w_{k}) } }{\varepsilon^{\frac{\ell+k}{\beta} - 1}(\alpha_{\ell+k,2}-\sigma_k(2))}\bigg)\,,\qquad \frac52 < k < {3}.
\end{equation}
Estimates \eqref{L2:relax}-\eqref{Hom-H-relax} gives the theorem.\end{proof}
As stated in the introduction, the convergence in Theorem \ref{theo:Sobolev} allows to deduce the convergence in mere $L^{1}$-norm by simple interpolation:
\begin{cor}\label{cor:L1conv} Let $g(t)=g(t,x)$ be a solution to the Boltzmann problem \eqref{eq:IB-selfsim}-\eqref{eq:normalisation-g0} with initial condition $g_{0}(x)=g(0,x)$ satisfying 
$$|\widehat{g_0}(\xi) | \leq \left(1+c^{2}\,|\xi|^{2}\right)^{-\frac{\beta}{2}}, \qquad \xi \in \R$$ for some $\beta,\,c>0$.  Then, for any $\frac{5}{2} < k < {3}$ and  $0<\sigma < \frac98 - \frac k4  - 2^{\frac{3}{2} - k}$, if $g_{0} \in L^{2}(\R) \cap L^{1}(\w_{k})$, there exists $C_{0} >0$ such that
$$\|g(t)-\bm{H}\|_{L^{1}} \leq C_{0}\exp\left(-\frac{4}{5}\sigma t\right), \qquad \forall t \geq0.$$
Moreover, if $g_{0} \in H^{\ell}(\R) \cap L^{1}(\w_{k})$  {for some $\ell>\frac12$}, there exists $C_{1} >0$ such that
$$\|g(t)-\bm{H}\|_{L^\infty} \leq C_{1}\exp\left(-\sigma t\right), \qquad \forall t \geq0.$$ \end{cor}
\begin{proof} By simple interpolation (see \cite[Theorem 4.2]{CGT}) one has
\begin{equation*}
\| g(t) - \bm{H} \|_{L^1} \leq C\,M_{2}(g(t) - \bm{H})^{\frac15}\,\| g(t) - \bm{H}\|^{\frac45}_{L^2}\,,
\end{equation*}
and the decay of the kinetic energy together with Theorem \ref{theo:Sobolev} with $\ell=0$ imply  the exponential relaxation towards equilibrium in the $L^{1}$ topology. Similarly, the exponential convergence in $L^{\infty}$ with rate $\sigma$ is shown by taking $\ell>\frac12$ and using Sobolev embedding.\end{proof}

\section{Linear analysis: spectral gap estimates}\label{Sec:linear}

The scope of this section is to derive spectral gap estimates for the linearized operator $\mathscr{L}$ defined in Definition \ref{def:L} in various functional spaces. Such spectral gap estimates amount to prove the exponential decay of solutions to  \eqref{eq:linearised} which we recall here to be
$$\p_t h = \mathscr{L}h = -\frac14 \p_x( x h) + 2 \Q_{0}(h, \bm{H})\,.$$ 
We begin with estimates in spaces defined by the Fourier norms \eqref{eq:fourier:norm:1} and \eqref{def:normkp}.

\subsection{Spectral gap in Fourier norms -- Proof of Theorem~\ref{specgap}}
\label{sec:sg} The proof of the existence of a spectral gap for the linearized operator $\mathscr{L}$ in norms $\vertiii{\cdot}_{k}$ and $\vertiii{\cdot}_{k,p}$ can be established following  exactly the lines of the proof of Theorem \ref{k-norm-cvgce} and turns out to be simpler so we just describe the main steps of it.

 \begin{proof}[Proof of Theorem~\ref{specgap}]
One directly sees that,  {under the normalisation \eqref{eq:normalisation-lin}}, the equation \eqref{eq:linearised}
preserves mass, momentum and energy. Notice also that, for $h$ satisfying \eqref{eq:normalisation-lin}, 
$$2\Q_{0}(h,\bm{H})(x)=4\left(h\ast \bm{H}\right)(2x)-h(x), \qquad x \in \R.$$
If $h$ is a solution to
\eqref{eq:linearised} and $\psi=\psi(t,\xi)$ is its  Fourier transform, then $\psi(t,\xi)$ satisfies the equation
\begin{equation}
  \label{eq:linearised-fourier}
  \p_t \psi(t,\xi) = \frac14 \xi  \p_\xi \psi(t,\xi)
  + 2 \psi\Big(t, \frac{\xi}{2} \Big) \bm{\Phi}\Big( \frac{\xi}{2} \Big)
  - \psi(t,\xi),
\end{equation}
which corresponds of course to the linearisation of
\eqref{eq:selfsim-fourier} around $\bm{\Phi}$ as given in \eqref{eq:Phi}. With the notations used in the proof of Theorem \ref{k-norm-cvgce}, we deduce from  Duhamel's formula that
\begin{equation}
  \label{eq:lin-Duhamel}
  \psi(t) = T(t) \psi_0 + \int_0^t T(t-s) B(s) \d s,
\end{equation}
where now 
\begin{equation*}
  B(s)=B(s,\xi) := 2 \psi\Big(s, \frac{\xi}{2} \Big)
  \bm{\Phi}\Big(\frac{\xi}{2} \Big).
\end{equation*}
Similarly to our calculation in Section \ref{Sec:conv:eq:fourier} we have $|B(s,\xi)| \leq 2 \left|\psi\left(s, \frac{\xi}{2}\right)\right|$ so that
\begin{equation*}
\vertiii{B(s)}_{k}
  \leq
  2^{1-k} \vertiii{\psi(s)}_k,
\end{equation*}
and we can use again \eqref{eq:decay1} to obtain, as in the proof of Theorem~\ref{k-norm-cvgce} that
\begin{equation*}\begin{split}
  \vertiii{\psi(t)}_{k} &\leq
  \vertiii{T(t) \psi_0}_k + \int_0^t \vertiii{T(t-s) B(s)}_k \d s
  \\
  &\leq
  \exp\left(-(1 - \frac14 k) t\right)\vertiii{\psi_0}_k
  +
  2^{1-k} \int_0^t \exp\left(-(1 - \frac14 k) (t-s)\right) \vertiii{ \psi(s) }_k \d s.
\end{split}\end{equation*}
As in Theorem \ref{k-norm-cvgce}, Gronwall's lemma allows to derive the exponential convergence in the norm $\vertiii{\cdot}_{k}$. The convergence in the norm $\vertiii{\cdot}_{k,p}$ follows exactly the same lines as in Theorem \ref{k-norm-cvgce}. \end{proof}
 
\subsection{Spectral gap in smaller spaces -- proof of Theorem~\ref{restrict}}
\label{sec:sg-small}

This section is devoted to the proof of Theorem~\ref{restrict} which restricts the spectral gap from Theorem~\ref{specgap} to the more tractable sub-space $\mathbb{Y}_a^{0}$,  {with $a\in(2,3)$}. For the proof, as explained in the introduction, we resort  results from \cite{GMM} and \cite{CanizoThrom} and use a suitable splitting of the linearised operator as 
$$\mathscr{L}=A+B,$$ with 
$$A: X_k \to \mathbb{Y}_{a}^{0} \qquad \text{ bounded for any }  { k>2} $$
and $B$ enjoying some dissipative properties. More precisely, for $R>1$ we consider nonnegative functions $\rho_R$  {and} $\theta_R\in {\mathcal C}^\infty(\R)$ which are bounded by $1$ and satisfy 
$$\theta_R(x)=\rho_{R}(x)=1 \qquad x \in \left[-\frac{R}{2},\frac{R}{2}\right]$$ 
and
$$\theta_R(x)=0 \quad \text{ for $|x|\geq \frac{R}{2}+1$}, \qquad  \rho_R(x)=0  \quad \text{ for $|x|\geq \frac{2}{3}R$.}$$ Let us now introduce the normalised Maxwellian 
$$\M(x)=\dfrac{e^{-x^2}}{\sqrt{\pi}}, \qquad x\in \R$$ and
$$
\zeta_1(x)=\left(\frac{3}{2} -x^2\right) \M(x), \qquad \zeta_2(x)=2x\,\M(x), \qquad
\zeta_3(x)=(-1+2x^2)\,\M(x).$$ We then define a bounded operator $\P :L^1(\w_{a})\to L^1(\w_{a})$ by 
\begin{equation}\label{projection}
\P h(x)= \zeta_{1}(x)\int_\R h(y)\dy \; +\zeta_{2}(x)\int_\R h(y) \, y \, \dy \;+\zeta_{3}(x)\int_\R h(y)\, y^2\, \dy\;,\qquad x \in \R. 
\end{equation}
For any $f\in L^1(\w_{a})$, one easily checks that 
$$f-\P(f)\in \mathbb{Y}_{a}^{0}.$$ 
Let us split $\mathscr{L}$ as $\mathscr{L}=A+B$ with 
$$A=A_1+A_2, \qquad \text{ and  } \qquad B=B_1+B_{2}+B_3,$$ 
where 
\begin{equation*}\begin{cases}
A_1h(x)&=  4\theta_R(x)\,((h\rho_{R})\ast \bm{H})(2x),   \qquad  A_2h=-\P(A_1h), \\
\\
  B_1h(x)&= -\frac{1}{4}\partial_x(xh)-h, \qquad  B_{3}h=\P(A_{1}h)
\end{cases}\end{equation*}
and $B_{2}=B_{2,1}+B_{2,2}$ with
$$B_{2,1}h(x)=  4(1-\theta_R(x))\,((h\rho_{R})\ast \bm{H})(2x), \qquad B_{2,2}h(x)= 4((h(1-\rho_{R}))\ast \bm{H})(2x).$$
Recalling that
$$\mathscr{L}h(x)=-\frac{1}{4}\partial_{x}(xh(x))-h(x)+4\left(h \ast \bm{H}\right)(2x)$$
for any $h$ satisfying \eqref{eq:normalisation-lin}, one sees that, indeed, $A+B=A_{1}+A_{2}+B_{1}+B_{2,1}+B_{2,2}+B_{3}=\mathscr{L}.$
The main property of $B=B_{1}+B_{2}+B_{3}$ is established in the following

\begin{prp}\phantomsection\label{prp:dissi}
 {Let $a\in(2,3)$.} Then, for any $0\leq \nu <1-\dfrac{a}{4} -  {2^{1-a}}$, the operator $B+\nu$ is dissipative in $L^{1}(\w_{a})$, i.e.
$$\int_{\R}Bh(x)\mathrm{sign}(h(x))\w_{a}(x)\dx \leq -\nu \int_{\R}|h(x)|\w_{a}(x)\d x, \qquad \forall h \in \mathscr{D}(\mathscr{L}) \subset L^{1}(\w_{a}).$$
\end{prp}

This proposition is a direct consequence of the following three lemmas.

\begin{lem}\phantomsection
For any $h \in \mathscr{D}(\mathscr{L}) \subset L^{1}(\w_{a})$, 
\begin{equation}\label{B1}
\int_{\R} B_1h(x) \, \mathrm{sign}(h(x))\w_{a}(x) \d x\leq \frac{a-4}{4}\int_{\R} |h(x)| \w_{a}(x) \d x.
\end{equation}
\end{lem}

\begin{proof}
Since $B_1h=-\dfrac{1}{4}\;x \partial_xh-\dfrac{5}{4} \;h$, an integration by parts leads to  
\begin{equation*}\begin{split}
\int_{\R} B_1h(x) \ &\mbox{sign}(h(x))\,\w_{a}(x) \d x\\
&= \frac{1}{4} \int_{\R}|h(x)|\, (\w_a(x)+a|x|\w_{a-1}(x)) \d x- \frac{5}{4} \int_{\R}|h(x)|\, \w_{a}(x) \d x\\
&\leq -\int_{\R} |h(x)|\, \w_{a}(x) \d x+\frac{a}{4} \int_{\R} |h(x)| \,\w_{a}(x) \d x,
\end{split}\end{equation*}
since $|x|\w_{a-1}(x)\leq \w_{a}(x)$
and \eqref{B1} follows.
\end{proof}

\begin{lem}\phantomsection
For any $a\in(2,3)$ and any $\varepsilon>0$, there exists $R>1$ such that for any $h\in L^{1}(\w_{a})$, 
\begin{equation}\label{B2}
\begin{split}
 \int_{\R} |B_{2,1}h(x)|\, \w_{a}(x)\d x &\leq \varepsilon  \int_{\R} |h(x)|\,\w_{a}(x) \d x\\
   \int_{\R} |B_{2,2}h(x)|\, \w_{a}(x)\d x &\leq (2^{1-a} +\varepsilon)  \int_{\R} |h(x)|\,\w_{a}(x) \d x\,.
\end{split}
\end{equation}
\end{lem}

\begin{proof}

We start with $B_{2,2}$ and a change of variables leads to 
$$ \int_{\R} |B_{2,2}h(x)| \,\w_{a}(x)\d x   = 2^{1-a}\int_{\R} |((h(1-\rho_R))\ast \bm{H})(x)| (2+|x|)^{a} \d x. $$
Now, since $\rho_R=1$ on $\left(-\frac{R}{2}, \frac{R}{2}\right)$, we get 
\begin{eqnarray*}
&  &\int_{\R} |B_{2,2}h(x)| \,\w_{a}(x) \d x \le 2^{1-a}\int_{\R} \int_{|y|\ge \frac{R}{2}} |h(y)| \bm{H}(x-y) (2+|x|)^{a} \d y \d x\\
& & \quad \le 2^{1-a}\int_{\R} \int_{|y|\ge \frac{R}{2}} \bm{H}(x-y) \w_{a}(x-y) |h(y)| \w_{a}(y) \left(\frac{1+|x-y|+1+|y|}{(1+|x-y|)(1+|y|)}\right)^{a} \d y \d x\\
& & \quad \le  2^{1-a}\int_{\R} \int_{|y|\ge \frac{R}{2}} \bm{H}(x-y) \w_{a}(x-y) |h(y)| \w_{a}(y) \left(\frac{1}{1+|y|}+\frac{1}{1+|x-y|}\right)^{a} \d y \d x.
\end{eqnarray*}
For $(x,y)\in\R^2$ with $|y|\ge \frac{R}{2}$, we have 
$$\frac{1}{1+|y|}+\frac{1}{1+|x-y|} \le \frac{2}{2+R} + \frac{\mathds{1}_{|x-y|\ge\frac{R}{2}}}{1+|x-y|}+ \frac{\mathds{1}_{|x-y|<\frac{R}{2}}}{1+|x-y|}\le \frac{4}{2+R} + 1.$$
Therefore, 
$$ \int_{\R} |B_{2,2}h(x)| \,\w_{a}(x) \d x \le  2^{1-a} \left( 1+\frac{4}{2+R} \right)^a \|\bm{H}\|_{L^1(\w_{a})} \|h\|_{L^1(\w_a)}.$$
For $R> 2$, one has $\frac{4}{2+R} \in(0,1)$ and we deduce from the inequality $(1+x)^a\le 1+ a 2^{a-1} x$, valid for any $x\in(0,1)$ and $a\in (2,3)$ that 
$$\int_{\R} |B_{2,2}h(x)| \,\w_{a}(x) \d x
\le \left(2^{1-a} +\frac{4a}{2+R}\right) \|\bm{H}\|_{L^1(\w_{a})} \|h\|_{L^1(\w_a)}.
$$
Choosing $R$ large enough, we obtain the second estimate in \eqref{B2}.

For the first bound in \eqref{B2} we proceed similarly and first change variables and use the properties of the cutoff functions to get
\begin{equation*}\begin{split}
 \int_{\R}|B_{2,1}h(x)|\w_a(x)\dx &= 2\int_{\R}|((h\rho_R)\ast \bm{H})(x)|\left(1-\theta_R\left(\frac{x}{2}\right)\right)\w_a\left(\frac{x}{2}\right)\dx\\*
& \leq 2\int_{|x|\geq R}\int_{\R}|h(y)|\rho_R(y)\bm{H}(x-y)\left(1+\abs{\frac{x}{2}}\right)^{a}\dy\dx\\*
& \leq 2\int_{|x|\geq R}\int_{|y|\leq \frac{2}{3}R}|h(y)|\bm{H}(x-y)\left(1+|x-y|+|y|\right)^{a}\dy\dx\\
& \leq 2\int_{|x|\geq R}\int_{|y|\leq \frac{2}{3}R}|h(y)|\bm{H}(x-y)\w_{a}(y)\w_{a}(x-y)\dy\dx\,.
\end{split}\end{equation*}
We next exploit that $\bm{H}\in L^1(\w_{\frac{3+a}{2}})$ for $a<3$ and $|x-y|\geq |x|-|y|\geq \frac{R}{3}$ for $|x|\geq R$ and $|y|\leq \frac{2R}{3}$ to deduce
\begin{multline*}
 \int_{\R}|B_{2,1}h(x)|\w_a(x)\dx \\
\le 2 \int_{|y|\leq \frac{2}{3}R}  |h(y)|\w_{a}(y) \int_{|x-y|\geq \frac{R}{3}} \bm{H}(x-y) \w_{a}(x-y) \dx \dy \\
\le 2 \|h\|_{L^1(\w_{a})} \int_{|x|\geq \frac{R}{3}} \bm{H}(x) \w_{a}(x) \dx
 \leq C\Bigl(1+\frac{R}{3}\Bigr)^{-\frac{3-a}{2}}\int_{\R}|h(y)|\w_a(y)\dy.
\end{multline*}
Since $2<a<3$, the first estimate in \eqref{B2} follows if we choose $R$ sufficiently large.
\end{proof}

\begin{lem}\phantomsection
For any $a\in (2,3)$ and any $\varepsilon>0$, there exists $R>1$ such that 
\begin{equation}\label{B3}
 \int_{\R} |B_3h(x)|\, \w_{a}(x)\d x \leq \varepsilon  \int_{\R} |h(x)|\, \w_{a}(x) \d x \qquad \forall h \in \mathscr{D}(\mathscr{L}) \subset \mathbb{Y}_{a}^{0}.
\end{equation}
\end{lem}

\begin{proof} Recall that $B_3h\,= \P(A_1h)$. Let us compute the first moments of $A_1h$. Using that $h\in \mathbb{Y}_{a}^{0}$ and that $\bm{H}$ has mass $1$,  momentum $0$ and energy $1$, one obtains 
\begin{equation*}\begin{split}
\int_\R A_1h(x)\dx & =   {2\int_\R h(x-y)\rho_{R}(x-y) \,\int_\R \bm{H}(y)\theta_R\left(\frac{x}{2}\right)\dx\dy}\\
& =  - {2}\int_{\R^2} h(x)\bm{H}(y) \left[1-\rho_R(x)\theta_R\left(\frac{x+y}{2}\right)\right]\dy\dx,
\end{split}\end{equation*}
\begin{equation*}\begin{split}
\int_\R A_1h(x)\ x \dx & =   {2} \int_\R h(x-y)\rho_R(x-y)\, \int_\R \frac{x}{2}\theta_R\left(\frac{x}{2}\right) \;\bm{H}(y) \dx\dy\\
&= -  \int_{\R^2} h(x)\bm{H}(y)\left[1-\rho_R(x)\theta_R\left(\frac{x+y}{2}\right)\right]\, (x+y) \dy\dx,
\end{split}\end{equation*}
and
\begin{equation*}\begin{split}
\int_\R A_1h(x)\ x^2 \dx & =   {2}\int_\R h(x) \rho_R(x)\, \int_{\R} \left(\frac{x+y}{2}\right)^2\theta_R\left(\frac{x+y}{2}\right) \;\bm{H}(y)\dx\dy \\
& =   -  {\frac{1}{2}}\int_{\R^2} h(x)\bm{H}(y)\left[1-\rho_R(x)\theta_R\left(\frac{x+y}{2}\right)\right](x+y)^2  \dy\dx.
\end{split}\end{equation*}
Consequently, one easily gets that
$$\left|B_{3}h(\cdot)\right| \leq   2\left(\int_{\R^2}|h(x)|\bm{H}(y)\left|1-\rho_R(x)\theta_{R}\left(\frac{x+y}{2}\right)\right|\w_{2}(x+y)\dy\dx\right)\sum_{i=1}^{3}\left|\zeta_{i}(\cdot)\right|$$
since $\max\left(1,|z|,(1+|z|^2)\right)\leq \w_{2}(z)$ and thus
\begin{equation*}
 \left|B_{3}h(\cdot) \right| \leq   2\left(\int_{\R^2}|h(x-y)|\bm{H}(y)\left|1-\rho_R(x-y)\theta_{R}\left(\frac{x}{2}\right)\right|\w_{2}(x)\dy\dx\right)\sum_{i=1}^{3}\left|\zeta_{i}(\cdot)\right|.
\end{equation*}
We next use the properties of the cutoff functions $\theta_R$ and $\rho_R$ together with $\w_s(x) \leq \w_s(y)\w_s(x-y)$ for $s\in\{a,2\}$ to deduce that
{\begin{equation*}\begin{split}
      \abs{1-\rho_R(x-y)\theta_R\Bigl(\frac{x}{2}\Bigr)}\w_2(x)
      &\leq \bm{1}_{\{|x-y|\geq \frac{R}{2}\}}\frac{\w_a(x-y)\w_a(y)}{\w_{a-2}(x-y)\w_{a-2}(y)} + \bm{1}_{\{|x|\geq R\}} \frac{\w_a(x)}{\w_{a-2}(x)} \\*      
 &\leq \frac{2^{a-2}}{(2+R)^{a-2}}\w_{a}(x-y)\w_{a}(y) + \frac{1}{(1+R)^{a-2}}\w_{a}(x)\\*      
 &\leq \frac{C}{(1+R)^{a-2}}\w_{a}(x-y)\w_{a}(y).
\end{split}\end{equation*}}
This yields
\begin{equation*}\begin{split}
\|B_3h\|_{L^{1}(\w_{a})}
 &\leq \frac{C}{ {(1+R)}^{a-2}}\left(\int_{\R}\bm{H}(y)\w_a(y)\int_{\R} |h(x-y)| {\w_a}(x-y)\dx\dy\right)\sum_{i=1}^{3} \|\zeta_i\|_{L^1(\w_{a})} \\
 &\leq  \frac{C}{ {(1+R)}^{a-2}} \|h\|_{L^1(\w_{a})},
\end{split}\end{equation*}
for some contant $C>0$ where we also used $\bm{H}\in L^{1}(\w_{a})$. We then deduce that \eqref{B3} holds provided $R$ is large enough.
\end{proof}
\begin{proof}[Proof of Proposition \ref{prp:dissi}]
The proof follows directly from the combination of \eqref{B1}--\eqref{B2}--\eqref{B3} since it implies that, for any $\varepsilon >0$, one can choose $R >1$ large enough so that
$$\int_{\R}Bh(x)\mathrm{sign}(h(x))\w_{a}(x)\dx \leq -\left(1-\frac{a}{4}-2^{1-a}-3\varepsilon\right)\|h\|_{L^{1}(\w_{a})} \qquad \forall h \in \mathscr{D}(\mathscr{L}) \subset \mathbb{Y}_{a}^{0}$$
which gives the result choosing $\varepsilon >0$ small enough so that $\nu=1-\frac{a}{4}-2^{1-a}- {3\varepsilon} \geq0.$
\end{proof}
We establish now the regularising effect of $A$:
\begin{prp}\label{prp:bounded}
Let $2<a<3$. The operator $A:X_k\to \mathbb{Y}_a^{0}$ is bounded  {for any $k>2$}. 
\end{prp}

This proposition follows directly from the following two lemmas. 

\begin{lem}\phantomsection \label{bound_A1}
Let $2<a<3$. There exists some constant $C>0$ such that, for any $h\in X_k$
$$\|A_1h\|_{L^1(\w_{a})} \le {C(k,R) \vertiii{\widehat{h}}_{k}}$$
for any $k>2$.
\end{lem}

\begin{proof}
First, one observes as before that
\begin{multline*}
\|A_1h\|_{L^1(\w_{a})} \leq  2\int_{\R} |((h\rho_R)\ast \bm{H})(x)|\theta_R\left(\frac{x}{2}\right) \,\w_{a}\left(\frac{x}{2}\right) \d x\\
\leq  2\int_{\R} |((h\rho_R)\ast \bm{H})(x)|\theta_R\left(\frac{x}{2}\right) \,\w_{a}(x) \d x\end{multline*}
where we used that $\w_{a}\left(\frac{x}{2}\right) \leq \w_{a}(x)$. We then deduce from the Cauchy-Schwarz inequality that 
$$
\|A_1h\|_{L^1(\w_{a})}
\leq 2\left(\int_{\R} |((h\rho_R)\ast \bm{H})(x)|^2\,\theta^2_R\left(\frac{x}{2}\right) \w_{a}(x)^2\,(1+|x|)^{2\chi} \d x\right)^{\frac{1}{2}}\left( \int_{\R} \frac{\d x}{(1+|x|)^{2\chi}} \right)^{\frac{1}{2}}
$$
with $\chi>\frac{1}{2}$. Thus, it holds 
$$\|A_1h\|_{L^1(\w_{a})} \leq 2\|\w_{-\chi}\|_{L^{2}} \,\|((h\rho_R)\ast \bm{H})\theta_R\left(\frac{\cdot}{2}\right)\,\w_{a+\chi}\|_{L^2} $$
Since $\theta_R {\left(\frac{x}2\right)}=0$ for $|x|\ge R+2$, we have 
$$\theta_R\left(\frac{x}{2}\right) \w_{a+\chi}(x) \le (3+R)^{a+\chi},$$
and 
\begin{equation}\label{sepA}
\|A_1h\|_{L^1(\w_{a})} \leq 2\, (3+R)^{a+\chi} \|\w_{-\chi}\|_{L^{2}} \, \left\|((h\rho_R)\ast \bm{H})\right\|_{L^2}.
\end{equation} 
 We deduce from the properties of the Fourier transform that  
$$\|(h\rho_R)\ast \bm{H} \|_{L^2}= \frac{1}{\sqrt {2\pi}}\; \|\widehat{(h\rho_R)\ast \bm{H}} \|_{L^2}  = \frac{1}{\sqrt {2\pi}}\;\|\widehat{h\rho_R}\cdot \widehat{\bm{H}} \|_{L^2}= \frac{1}{(2\pi)^{\frac{3}{2}}} \; \|(\widehat{h}\ast \widehat{\rho_R})\, \widehat{\bm{H}} \|_{L^2}.$$
We have $|\eta|^k\leq (|\xi-\eta|+|\xi|)^k\leq \w_{k}(\xi-\eta)\w_{k}(\xi)$. Thus, 
\begin{equation*}\label{eq:convolution:fourier}
|(\widehat{h}\ast \widehat{\rho_R})(\xi) |
\leq   {\vertiii{\widehat{h}}_{k}} \int_{\R} |\eta|^k \, |\widehat{\rho_R}(\xi-\eta)| \d \eta \leq   {\vertiii{\widehat{h}}_{k}}\, \w_{k}(\xi)\,  \|\widehat{\rho_R}\|_{L^1(\w_{k})}. 
\end{equation*}
Hence, 
$$\|(h\rho_R)\ast \bm{H} \|_{L^2} \leq  \frac{ {1}}{(2\pi)^{\frac{3}{2}}}\; {\vertiii{\widehat{h}}_{k}} \, 
\| \widehat{\rho_R} \|_{L^1(\w_k)} \, \| \w_{k}(\cdot)\widehat{\bm{H}}\|_{L^2}\,.$$
Since $\rho_R \in {\mathcal C}^\infty(\R)$ is compactly supported, $\widehat{\rho_R}\in L^1(\w_{k})$ for any {$k>2$}. Furthermore, $\w_{k}(\xi)\widehat{\bm{H}}(\xi)= (1+|\xi|)^{k+1}e^{-|\xi|} \in L^2(\R)$  for any  {$k>2$}. Consequently, there exists some constant $C_1(k,R)>0$ such that 
\begin{equation}\label{A1}
\|(h\rho_R)\ast \bm{H} \|_{L^2} \leq  C_1(k,R)  {\vertiii{\widehat{h}}_{k}}.
\end{equation}
Gathering \eqref{sepA} and \eqref{A1} completes the proof. 
\end{proof}

\begin{lem}\phantomsection
Let $2<a<3$. There exists some constant $C>0$ such that, for any $h\in X_k$
$$\|A_2h\|_{L^1(\w_{a})} \leq  C {\vertiii{\widehat{h}}_{k}}$$
for any $k>2$.
\end{lem}

\begin{proof}
It follows from the definition \eqref{projection} of $\P$ that 
\begin{equation*} 
\|A_2h\|_{L^1(\w_{a})} \leq  2 \int_\R |A_1h(x)| \, (1+x^2)\dx \max_{i\in\{1,2,3\}}  \|\zeta_i\|_{L^1(\w_{a})} 
\leq  C \|A_1h\|_{L^1(\w_{a})}, 
\end{equation*}
and the result follows from Lemma \ref{bound_A1}.
\end{proof}

 \begin{proof}[Proof of Theorem~\ref{restrict}]
 The existence of a spectral gap for $\mathscr{L}$ in $\mathbb{Y}_{a}^{0}$ is now a direct consequence of Propositions~\ref{prp:dissi} and \ref{prp:bounded} together with \cite[Theorem 5.2]{CanizoThrom}.
\end{proof}
  
\subsection{Additional consequences}\label{Sec:additional}

We finally establish some useful consequences of the above spectral gap estimates. Such consequences are particularly relevant for the study of the self-similar profiles associated to the $1D$ inelastic Boltzmann equation with moderate hard potentials as studied in the companion paper \cite{unique-short}. We believe that such results have their own interest and pertain to the present contribution since they only concern the solutions to the one-dimensional Boltzmann equation with Maxwell molecules. To motivate such results, we however briefly recall the definition of the collision operator associated with moderate hard potentials as studied in \cite{unique-short}. For $\g \in (0,1)$, we consider the collision operator $\Q_{\g}(f,g)$  given in weak form by
\begin{equation}\label{eq:weakgamma}
\int_{\R}\Q_{\g}(f,g)(x)\phi(x)\d x= {\frac12}\int_{\R^{2}}f(x)g(y)\left(2\phi\left(\frac{x+y}{2}\right)-\phi(x)-\phi(y)\right)|x-y|^{\g}\d x\d y\end{equation}
for any smooth enough test function $\phi=\phi(x)$. Notice that, for $\g=0,$ one recovers the expression \eqref{eq:weak} of $\Q_{0}.$ The solutions to the associated evolution problem
$$\partial_{t}f=\Q_{\g}(f,f)$$
are still dissipating energy and \eqref{eq:Ener} reads now
$$\dfrac{\d}{\d t}E(t)= -\frac{1}{4}\int_{\R^{2}}f(t,x)f(t,y)|x-y|^{2+\g}\d x\d y=-\frac{1}{4}\mathscr{D}_{\g}(f(t))\,.$$
One notices then that, in a regime $\g\simeq 0$, one expects 
$$|x-y|^{2+\g} \simeq |x-y|^{2}\left(1+\g\log|x-y|\right)$$
and, after linearising the above dissipation of energy $\mathscr{D}_{\g}(f)$ around a suitable self-similar profile, one is naturally led to the study of the \emph{``linearised dissipation of energy''} functional $\mathscr{I}_{0}(f,\bm{G}_{0})$ where $\bm{G}_{0}$ is a steady solution to \eqref{eq:IB-selfsim} with positive energy $E_{0}$ and with
\begin{equation}\label{eq:mathIo}
\mathscr{I}_{0}(f,g)=\int_{\R^{2}}f(x)g(y)|x-y|^{2}\log|x-y|\dx\dy, \qquad f,g \in L^{1}(\w_{s}), \qquad s>2.\end{equation} 
We refer to \cite{unique-short} for details. Since $\bm{G}_{0}$ is a steady solution to \eqref{eq:IB-selfsim}, according to Theorem \ref{theo:bob}, there exists $\lambda_{0} >0$ such that
$$\bm{G}_{0}(x)=\lambda_{0}\bm{H}(\lambda_{0} x)\,.$$
It is therefore particularly interesting to understand the link between the linearised operator $\mathscr{L}$ given by Definition~\ref{def:L} and the functional $\mathscr{I}_{0}$ defined in \eqref{eq:mathIo}. First of all, an important point is to prove that the linearised operator around $\bm{G}_{0}$ instead of $\bm{H}$ still enjoys the same properties of $\mathscr{L}$. Precisely, one defines
$$\mathscr{L}_{0}:\mathscr{D}(\mathscr{L}_{0}) \subset L^{1}(\w_{a})\to  L^{1}(\w_{a})$$
by {$$\mathscr{L}_{0}(h)=2\Q_{0}(h,\bm{G}_{0})-\frac{1}{4}\partial_{x}(xh), \qquad \forall h \in \mathscr{D}(\mathscr{L}_{0})$$}
with 
$$\mathscr{D}(\mathscr{L}_{0})=\left\{f \in  L^{1}(\w_{a})\;;\;\partial_{x}(x f) \in  L^{1}(\w_{a})\right\}$$
and $\bm{G}_0=\lambda_0\bm{H}(\lambda_0 \cdot)$ with $\lambda_0>0$. By a simple scaling argument, the result from Theorem~\ref{restrict} can be transferred to $\mathscr{L}_0$ :
\begin{prp}\label{restrict0}
  Let $2<a<3$. The operator $\left(\mathscr{L}_{0},\mathscr{D}(\mathscr{L}_{0})\right)$ on $ L^{1}(\w_{a})$ is such that, for any $ {\nu \in(0,1-\frac{a}{4} -2^{1-a})}$, there exists $C(\nu)>0$ such that
\begin{equation}\label{eq:invert0} 
 \|\mathscr{L}_{0}h\|_{ L^{1}(\w_{a})}\ge \frac{\nu}{C(\nu)} \|h\|_{ L^{1}(\w_{a})}, \qquad \forall\, h \in \mathscr{D}(\mathscr{L}_{0}) \cap \mathbb{Y}_{a}^{0}.\end{equation}
In particular, the restriction $\widetilde{\mathscr{L}}_{0}$ of $\mathscr{L}_{0}$ to the space $\mathbb{Y}_{a}^{0}$  
is invertible with
\begin{equation}\label{eq:invert}\left\|\widetilde{\mathscr{L}}_{0}^{-1}g\right\|_{ L^{1}(\w_{a})} \leq \frac{C(\nu)}{\nu}\|g\|_{ L^{1}(\w_{a})}, \qquad \forall\, g \in \mathbb{Y}_{a}^{0}.\end{equation}
\end{prp}
 \begin{proof}  We consider $a \geq k$ and the spaces $X_{k}$ and $L^{1}(\w_{a})$ defined previously so that $\mathbb{Y}_{a}^{0} \subset X_{k}.$ Notice that $\mathscr{D}(\mathscr{L}_{0})=\mathscr{D} {(\mathscr{L})} \cap L^{1}(\w_{a})$ and, since $\bm{G}_{0}(\cdot)=\lambda_{0}\bm{H}(\lambda_{0}\cdot)$, one checks easily that, for any test function $\phi$
$$\int_{\R}\mathscr{L}_{0}(f)(x)\phi(x)\d x=\frac{1}{\lambda_{0}}\int_{\R}\mathscr{L}(\tau_{0}{f})(x)\phi\left(\lambda_{0}^{-1}x\right)\d x=\int_{\R}\mathscr{L}(\tau_{0}{f})(\lambda_{0} y)\phi(y)\d y$$
where 
$$\tau_{0}{f}(x)=f\left(\frac{x}{\lambda_{0}}\right), \qquad x \in \R.$$
This shows that
$$\mathscr{L}_{0}f=\tau_{0}^{-1}\mathscr{L}\left(\tau_{0}f\right), \qquad \forall f \in \mathscr{D}(\mathscr{L}_{0}).$$
In particular, since $\mathbb{Y}_{a},\mathbb{Y}_{a}^{0}$ are invariant under the action of the bijective transformation $\tau_{0}$ and of course
$$\mathrm{Range}(\mathscr{L}_{0})=\mathrm{Range}(\mathscr{L})=\mathbb{Y}_{a}^{0}$$
one sees that $\mathbb{Y}_{a}^{0}$
is a closed linear subspace of $L^{1}(\w_{a})$ stable under $\mathscr{L}_{0}.$ This allows to define in a standard way  the restriction $\widetilde{\mathscr{L}}_{0}:=\mathscr{L}_{0}\vert_{\mathbb{Y}_{a}^{0}}$ of $\mathscr{L}_{0}$ to the space $\mathbb{Y}_{a}^{0}$ 
$$\widetilde{\mathscr{L}}_{0}=\mathscr{L}_{0}\vert_{\mathbb{Y}_{a}^{0}} \;\;:\;\;\mathscr{D}(\mathscr{L}_{0}) \cap \mathbb{Y}_{a}^{0} \to \mathbb{Y}_{a}^{0}$$
and one can deduce then from  Theorem \ref{restrict} the result.\end{proof}
 A first result to understand the connection between $\mathscr{L}$ and $\mathscr{I}_0$ is the following
\begin{lem}\label{rmk:kern} The function defined by
$$g_{0}(x)=\frac{2}{\pi}\frac{1-3x^{2}}{(1+x^{2})^{3}}, \quad x \in \R$$
belongs to $\mathbb{Y}_{a}$ and 
is such that
$$\mathscr{L}(g_{0})=0 \qquad \text{ and }\qquad  M_{2}(g_{0})=-2.$$
Moreover,
\begin{equation}\label{eq:g0H}
\mathscr{I}_{0}(g_{0},\bm{H})= -2\log 2-2.\end{equation}
Finally, it holds 
$$\mathscr{I}_{0}(\bm{H},\bm{H})=2\log 2+1.$$
\end{lem}
\begin{proof} Let $g \in L^{1}(\w_{a})$  {with $2<a<3$} be such that $\mathscr{L}(g)=0$ and $\ds\int_{\R} g(x)\d x=0.$ Setting
$$\psi(\xi)=\int_{\R}e^{-i\xi x}g(x)\d x \qquad  {\mbox{ and }  \qquad 
\bm{\Phi}(\xi)=\int_{\R}e^{-i\xi x} \bm{H}(x)\d x, } $$
one checks without too many difficulties that {(see also \eqref{eq:linearised-fourier})}
$$-\frac{1}{4}\xi \frac{\d}{\d \xi}\psi(\xi)=2\psi\left(\frac{\xi}{2}\right)\bm{\Phi}\left(\frac{\xi}{2}\right)-\psi(\xi).$$
Recalling that $\bm{\Phi}(\xi) = (1 + |\xi|) e^{-|\xi|},$ direct inspection shows that
$$\psi_{0}(\xi)=|\xi|^{2}e^{-|\xi|}$$
is a solution to the above equation, with 
\begin{equation}\label{eq:psi0}
\psi_{0}(0)=\psi'_{0}(0)=0, \quad \psi''_{0}(0)=2 \neq 0.\end{equation}
Moreover, since $e^{-|\xi|}$ is the Fourier transform of $G(x)=\frac{1}{\pi(1+x^{2})}$, one deduces that $\psi_{0}$ is the Fourier transform of 
$$g_{0}(x)=-\dfrac{\d^{2}}{\d x^{2}}G(x)= {\frac{2}{\pi}}\frac{1-3x^{2}}{(1+x^{2})^{3}}.$$
Notice that $g_{0} \in L^{1}(\w_{a})$ for any $2<a<3$ and \eqref{eq:psi0} shows that $g_{0} \in \mathbb{Y}_{a}$ with $M_{2}(g_{0})=-2.$ Let us now prove \eqref{eq:g0H}. Observe that, if $g$ is an eigenfunction of $\mathscr{L}$ with zero mass, then using the weak form of the linearised operator $\mathscr{L}$, $$ \frac14 \int_\R g(x) x \p_x \phi\,\dx
  + 2 \int_\R \int_\R g(x) \bm{H}(y) \left(
    \phi\Big(\frac{x-y}{2}\Big) - \frac12\,\phi(x) - \frac12\,{\phi(-y)}
  \right) \dy \d x=0\,,$$
where we used also that $\bm{H}$ is even. Taking $\phi(x)=x^{2}\log|x|=\frac{1}{2}x^{2}\log x^{2}$ as a test-function we get
\begin{multline*}
\frac{1}{8}\int_{\R}g(x)x\p_x(x^{2}\log x^{2})\,\dx+2\int_{\R^{2}}g(x)\bm{H}(y)\frac{|x-y|^{2}}{4}\log\frac{|x-y|}{2}\d x\d y\\
-\int_{\R}g(x)x^{2}\log |x|\d x=0\,,\end{multline*}
where we used that $\ds\int_{\R}g(x)\d x=0$ while $\ds\int_{\R}\bm{H}(y)\d y=1$. Thus, one obtains that any eigenfunction of $\mathscr{L}$ with zero mass is such that
\begin{equation}\label{eq:IoGG}\begin{split}
\mathscr{I}_{0}(g,\bm{H})&:=\int_{\R^{2}}g(x)\bm{H}(y)|x-y|^{2}\log|x-y|\dx\dy\\
&=\left(\log 2-\frac{1}{2}\right)\int_{\R}g(x)x^{2}\d x  { + } \int_{\R}g(x)x^{2}\log |x|\d x.\end{split}\end{equation} 
In particular, for $g=g_{0}=-\frac{\d^{2}}{\d x^{2}}G$ as defined previously, it holds that
\begin{equation*}\begin{split}
\int_{\R}g_{0}(x)x^{2}\log|x|\d x&=-\int_{\R}G(x)\dfrac{\d^{2}}{\d x^{2}}\left[x^{2}\log|x|\right]\dx=-\frac{1}{2}\int_{\R}G(x)\dfrac{\d^{2}}{\d x^{2}}\left[x^{2}\log x^{2}\right]\dx\\
&=-2\int_{\R}G(x)\log|x|\dx-3\int_{\R}G(x)\dx=-3\,,\end{split}
\end{equation*}
using $\ds\int_{\R}G(x)\dx=1$, and
$$\int_{\R}\frac{\log|x|}{1+x^{2}}\dx=2\int_{0}^{\infty}\frac{\log x}{1+x^{2}}\dx=0.$$
Therefore, recalling that $M_{2}(g_{0})= {-2}$, we deduce \eqref{eq:g0H}. The same idea gives also the expression of $\mathscr{I}_{0}(\bm{H},\bm{H})$. Indeed, by definition
\begin{equation*}\begin{split}
-\frac{1}{4}\int_{\R}x\bm{H}(x)\p_{x}\phi(x)\d x&=\int_{\R}\Q_{0}(\bm{H},\bm{H})\phi \dx\\
&=\int_{\R^{2}}\bm{H}(x)\bm{H}(y)\left[\phi\left(\frac{x+y}{2}\right)-\phi(x)\right]\dx \dy.\end{split}\end{equation*}
With $\phi(x)=|x|^{2}\log |x|$, this gives, since $\ds\int_{\R}\bm{H}(x)\dx=\ds\int_{\R}\bm{H}(x)x^{2}\dx=1$,
\begin{equation*}\begin{split}
-\frac{1}{2}\int_{\R}x^{2}\bm{H}(x)\log|x|\d x-\frac{1}{4}&=\frac{1}{4}\int_{\R^{2}}\bm{H}(x)\bm{H}(y)|x+y|^{2}\log|x+y|\dx\dy\\
& -\frac{\log 2}{4}\int_{\R^{2}}\bm{H}(x)\bm{H}(y)|x+y|^{2}\dx\dy-\int_{\R}\bm{H}(x)x^{2}\log|x|\dx\\
&=\frac{1}{4}\mathscr{I}_{0}(\bm{H},\bm{H})-\frac{\log2}{2}-\int_{\R}\bm{H}(x)x^{2}\log|x|\dx
\end{split}\end{equation*}
i.e.
$$\mathscr{I}_{0}(\bm{H},\bm{H})=2\log 2-1+2\int_{\R}\bm{H}(x)x^{2}\log |x|\dx.$$
Using that $\ds \int_{\R}\bm{H}(x)x^{2}\log|x|\dx=1$
we deduce the result.
\end{proof}
Thanks to the previous observations, we deduce the following lemma.
\begin{lem}\label{rmk:varphi0} Let 
$$\bm{G}_{0}(x)=\lambda_{0}\bm{H}(\lambda_{0}x) \qquad \text{with} \quad \lambda_{0}= {\exp\left(\frac{1}{2}\mathscr{I}_{0}(\bm{H},\bm{H})\right)}$$
and let $\mathscr{L}_{0}$ be the associated linearized operator in $L^{1}(\w_{a})$ with $2 < a < 3$.  There exists $\varphi_{0} \in \mathrm{Ker}(\mathscr{L}_{0}) \cap \mathbb{Y}_{a}$ such that
$$M_{2}(\varphi_{0}) \neq 0 \qquad \text{ and } \quad \mathscr{I}_{0}(\varphi_{0},\bm{G}_{0}) \neq 0.$$
\end{lem}
\begin{proof} Since the function $g_{0}$ defined in Lemma \ref{rmk:kern} belongs to the kernel of $\mathscr{L}$ and to $\mathbb{Y}_{a}$, one has
$$\varphi_{0}(x)=g_{0}(\lambda_{0}x) \in \mathbb{Y}_{a} \cap \mathrm{Ker}(\mathscr{L}_{0}).$$
Moreover, recalling the definition of $\mathscr{I}_{0}$ in \eqref{eq:mathIo}  and since $\bm{G}_{0}(x)=\lambda_{0}\bm{H}(\lambda_{0}x)$, one checks easily that
\begin{multline*}
\mathscr{I}_{0}(\varphi_{0},\bm{G}_{0})=\frac{1}{\lambda_{0}^{3}}\left(\mathscr{I}_{0}(g_{0},\bm{H})-\log \lambda_{0}\,\int_{\R^{2}}g_{0}(x)\bm{H}(y)|x-y|^{2}\d x\d y\right)\\
=\frac{1}{\lambda_{0}^{3}}\left(\mathscr{I}_{0}(g_{0},\bm{H})-\log \lambda_{0}M_{2}(g_{0})\right)\end{multline*}
where we used that $g_{0} \in \mathbb{Y}_{a}.$ In particular, since $M_{2}(g_{0})= {-2}$ and $\lambda_0= {\exp\left(\frac12 \mathscr{I}_{0}(\bm{H},\bm{H})\right)}$ , we deduce that
$$\mathscr{I}_{0}(\varphi_{0},\bm{G}_{0})=\frac{1}{\lambda_{0}^{3}}\mathscr{I}_{0}\left( {g_{0}+\bm{H}},\bm{H}\right)=-\frac{1}{\lambda^{3}_{0}} \neq 0$$
where we used that $\mathscr{I}_{0}(g_{0},\bm{H})= -2\log 2-2$ and $\mathscr{I}_{0}(\bm{H},\bm{H}) = 2\log 2+1$.\end{proof}
The existence of the above function $\varphi_{0}$ implies the following fundamental property of the linearised dissipation of energy: 
\begin{lem}\label{lem:roleI0}Let 
$$\bm{G}_{0}(x)=\lambda_{0}\bm{H}(\lambda_{0}x) \qquad \text{with} \quad \lambda_{0}={\exp\left(\frac{1}{2}\mathscr{I}_{0}(\bm{H},\bm{H})\right)}$$
and let $\mathscr{L}_{0}$ be the associated linearized operator in $L^{1}(\w_{a})$ with $2 < a < 3$. If $\varphi \in \mathrm{Ker}(\mathscr{L}_{0} ) \cap \mathbb{Y}_{a}$ then 
$$\mathscr{I}_{0}(\varphi,\bm{G}_{0})=0 \implies  M_{2}(\varphi)=0.$$
In particular, in such a case, $\varphi=0.$
\end{lem}
\begin{proof} Let $\varphi \in \mathrm{Ker}(\mathscr{L}_{0}) \cap \mathbb{Y}_{a}$ be such that $\mathscr{I}_{0}(\varphi,\bm{G}_{0})=0$. Let 
$$\varphi^{\perp}=\varphi-\frac{M_{2}(\varphi)}{M_{2}(\varphi_{0})}\varphi_{0}.$$
One has of course $M_{2}(\varphi^{\perp})=0$ (i.e. $\varphi^{\perp} \in \mathbb{Y}_{a}^{0}$) and $\mathscr{L}_{0} (\varphi^{\perp})=0$ since both $\varphi$ and $\varphi_{0}$ belong to $\mathrm{Ker}(\mathscr{L} )$. According to Proposition \ref{restrict0}, one has $\varphi^{\perp}=0$. Therefore,
$$\varphi=\frac{M_{2}(\varphi)}{M_{2}(\varphi_{0})}\varphi_{0} \qquad \text{ and consequently } \quad \mathscr{I}_{0}(\varphi,\bm{G}_{0})=\frac{M_{2}(\varphi)}{M_{2}(\varphi_{0})}\mathscr{I}_{0}(\varphi_{0},\bm{G}_{0}).$$
Since, by assumption $\mathscr{I}_{0}(\varphi,\bm{G}_{0})=0$ while $\mathscr{I}_{0}(\varphi_{0},\bm{G}_{0}) \neq 0$, it must hold that $M_{2}(\varphi)=0.$ In particular, $\varphi \in \mathbb{Y}_{a}^{0}$ and, using Proposition \ref{restrict0} again, we deduce that $\varphi=0.$ 
\end{proof}

\appendix

\numberwithin{equation}{section}

\section{Properties of the Fourier norm}\label{appendix}

The following lemma is a consequence of \cite[Lemma 2.5]{MR2355628}.

\begin{lem}\label{lem:mukvertk}
 Let $2<k<3$. There exists a constant $C>0$ depending only on $k$ such that
  $$\vertiii{\widehat{\mu}}_{k}\leq C \int_{\R}  (1+|x|)^{k} \,|\mu|(\d x),$$
 for any $\mu\in X_{k}$. Similarly, for any $1\le p <\infty$ and $2<k<3$, there exists a constant $C>0$ depending only on $k$ and $p$ such that
  $$\vertiii{\widehat{\mu}}_{k,p}\leq C \int_{\R}  (1+|x|)^{k} \,|\mu|(\d x),$$
 for any $\mu\in X_{k}.$
 
\end{lem}

\begin{proof}
  Since $\mu \in X_{k}$, we have $\widehat{\mu}(0)=0$, $\widehat{\mu}'(0)=0$ and $\widehat{\mu}''(0)=0$. Hence, Taylor formula implies that
  $$|\widehat{\mu}(\xi)|\le |\xi|^2 \int_0^1 | \widehat{\mu}''(t\xi)| \d t .$$
We set $s=k-2\in(0,1)$. Then, for ${\phi}(r)=r^s$, we have
  $$M:=\int_\R (1+x^2)\phi(|x|)|\mu|(\d x)<\infty.$$
  Moreover, $\phi$ is a strictly increasing function with $\frac{\phi(r)}{r}$ nonincreasing. It follows from \cite[Lemma 2.5]{MR2355628} that
  $$|\widehat{\mu}''(t\xi)|\le 2 M \psi(|t\xi|),$$
  where $\psi(y)=[\phi(y^{-1})]^{-1}=y^s$. Hence,
   \begin{equation}\label{bound}
|\widehat{\mu}(\xi)|\le  2M\, |\xi|^{2+s}   \int_0^1t^s\d t \le  \frac{2 M}{s+1} |\xi|^{2+s}. 
\end{equation}
This proves the first part of the result {since $s+2=k$ and $M\le\ds\int_{\R} (1+|x|)^k \,|\mu|(\dx)$.} Now, for the second part, given $1 \leq p < \infty$, we have 
$$\vertiii{\widehat{\mu}}_{k,p}^p = \int_{|\xi|\le 1} \frac{|\widehat{\mu}(\xi)|^p}{|\xi|^{kp}} \; \d\xi + \int_{|\xi|> 1} \frac{|\widehat{\mu}(\xi)|^p}{|\xi|^{kp}} \; \d\xi.   $$
Next, for $|\xi|> 1$, we simply use the bound $|\widehat{\mu}(\xi)| \le \ds\int_{\R} |\mu|(\d x)$  whereas, for $|\xi|\le 1$, we use the bound \eqref{bound}. This leads to 
  $$\vertiii{\widehat{\mu}}_{k,p}^p  \le \frac{2^{p+1}}{(k-1)^p}  \left(\int_{\R}  (1+|x|)^{k} \,|\mu|(\d x)\right)^p + \left(\int_{\R} |\mu|(\d x)\right)^p \int_{|\xi|>1} \frac{\d \xi}{|\xi|^{pk}}.$$
The result then follows since $\ds \int_{|\xi|>1} \frac{\d \xi}{|\xi|^{pk}}<\infty$. 
\end{proof}}

\bibliographystyle{plainnat-linked}

\begin{thebibliography}{ABBB}
 
\expandafter\ifx\csname urlstyle\endcsname\relax
  \providecommand{\doi}[1]{doi: #1}\else
  \providecommand{\doi}{doi: \begingroup \urlstyle{rm}\Url}\fi


\bibitem[Alonso et~al. (2024)]{unique-short}
\newblock \textsc{R.  Alonso, V.  Bagland, J.~A. Ca\~nizo, B.  Lods \& S. Throm},  
\newblock One-dimensional
  inelastic Boltzmann equation: Stability and uniqueness of
  self-similar $L^{1}$-profiles for moderately hard potentials, 
\newblock \textit{preprint (2024)}. 



\bibitem[Alonso et~al. (2018)]{ABCL}
\newblock \textsc{R.  Alonso, V.  Bagland, Y. Cheng \& B.  Lods},  
\newblock One-dimensional dissipative Boltzmann equation: measure solutions, cooling rate, and self-similar profile, 
\newblock {\it SIAM J. Math. Anal.}, {\bf 50} (2018), 1278--1321. 


\bibitem[Ben-Naim \& Krapivsky(2000)]{BK}
\newblock \textsc{E. Ben-Naim \& P. Krapivsky,}
\newblock {Multiscaling in inelastic collisions,}
\newblock \textit{Phys. Rev. E,} {\bf 61} (2000), 011309.

\bibitem[Bobylev \&  Cercignani(2003)]{bobcerc1}
\newblock\textsc{A. V. Bobylev \& C. Cercignani,}
\newblock {Self-similar asymptotics for the Boltzmann equation with inelastic and elastic interactions,}
\newblock \textit{J. Statist. Phys.} {\bf 110} (2003), 333--375.

\bibitem[Ca\~nizo \& Throm (2021)]{CanizoThrom}
\textsc{J.~A. Ca\~nizo, \& S. Throm}, The scaling hypothesis for Smoluchowski's coagulations equation with bounded perturbations of the constant kernel, \textit{J. Differential Equations} {\bf 270} (2021), 285--342.

\bibitem[Carlen et~al. (1999)] {CGT}
\textsc{E.~A. Carlen, E. Gabetta \& G.  Toscani}, Propagation of smoothness and the rate of exponential convergence to equilibrium for a spatially homogeneous Maxwellian gas. \textit{Comm. Math. Phys.}, {\bf 199} (1999), 521--546. 

\bibitem[Carrillo \&  Toscani (2007)]{MR2355628}
\textsc{J. A. Carrillo \& G.  Toscani},  Contractive probability metrics and asymptotic behavior of dissipative kinetic equations. \textit{Riv. Mat. Univ. Parma} \textbf{7} (2007), 75--198.

\bibitem[Gualdani et~al. (2017)]{GMM}
\textsc{M. P.  Gualdani, S. Mischler \& C.  Mouhot}, Factorization for non-symmetric operators and exponential H-theorem, \textit{M\'emoires de la SMF}, {\bf 153}, 2017.



\bibitem[Furioli  et~al. (2009)]{FPTT}
\textsc{G.  Furioli, A.  Pulvirenti, E. Terraneo, \& G.  Toscani}, Strong convergence towards self-similarity for one-dimensional dissipative Maxwell models, \textit{Journal of Functional Analysis}, {\bf 257} (2009),  2291--2324.


\bibitem[Mischler \& Mouhot (2016)]{MM}
\newblock \textsc{S. Mischler \& C. Mouhot}, 
\newblock Exponential stability of slowly decaying solutions to the kinetic-Fokker-Planck equation, 
\newblock \textit{Arch. Ration. Mech. Anal.} {\bf 221} (2016), 677--723.




\end{thebibliography}

\end{document}